\documentclass{amsart}

\usepackage{mathtools}
\usepackage{xypic}
\usepackage{amsthm}
\usepackage{amsmath}
\usepackage{amssymb}
\usepackage{mathrsfs}
\usepackage{cite}
\usepackage{graphicx}

\newtheorem{thm}{Theorem}
\newtheorem{question}[thm]{Question}
\newtheorem{lemma}[thm]{Lemma}
\newtheorem{prop}[thm]{Proposition}

\newtheorem{defn}[thm]{Definition}
\newtheorem{remark}[thm]{Remark}
\newtheorem{conj}[thm]{Conjecture}
\newtheorem{ex}[thm]{Example}

\newcommand{\nc}{\newcommand}
\nc{\rnc}{\renewcommand}
\rnc{\P}{\mathbb P}
\nc{\R}{\mathbb R}
\nc{\C}{\mathbb C}
\nc{\A}{\mathbb A}
\nc{\Q}{\mathbb Q}
\nc{\Z}{\mathbb Z}
\rnc{\O}{\mathcal O}
\nc{\LL}{\mathbb L}
\nc{\Hom}{\mathrm{Hom}}
\nc{\codim}{\mathrm{codim}}
\nc{\Sym}{\mathrm{Sym}}
\nc{\Spec}{\mathrm{Spec}\,}
\nc{\End}{\mathrm{End}}
\nc{\eps}{\epsilon}
\nc{\Pic}{\mathrm{Pic}}
\nc{\ov}{\overline}
\nc{\X}{\mathcal X}
\nc{\F}{\mathfrak F}
\nc{\G}{\mathbb G}
\nc{\E}{\mathcal E}
\rnc{\S}{\mathcal S}

\rnc{\L}{\mathcal L}
\rnc{\H}{\mathcal H}
\nc{\Bl}{\text{Bl}}

\nc{\us}{\underset}
\nc{\ul}{\underline}
\nc{\bs}{\backslash}
\nc{\os}{\overset}

\nc{\Mod}{\mathrm{Mod}}
\nc{\MW}{\mathrm{MW}}
\nc{\NS}{\mathrm{NS}}
\nc{\Res}{\mathrm{Res}}
\nc{\Aut}{\mathrm{Aut}}
\nc{\W}{\mathcal W}
\nc{\NL}{\mathrm{NL}}
\nc{\mult}{\mathrm{mult}}

\nc{\U}{\mathscr U}

\rnc{\v}{{\langle v \rangle}}

\title{Modular Forms from Noether-Lefschetz Theory}

\author{Fran\c{c}ois Greer}

\begin{document}

\maketitle
\begin{abstract}
We enumerate smooth rational curves on very general Weierstrass fibrations over hypersurfaces in projective space.  The generating functions for these numbers lie in the ring of classical modular forms.  The method of proof uses topological intersection products on a period stack and the cohomological theta correspondence of Kudla and Millson for special cycles on a locally symmetric space of orthogonal type.  The results here apply only in base degree 1, but heuristics for higher base degree match predictions from the topological string partition function.
\end{abstract}

\section{Introduction}
\noindent
Locally symmetric spaces of noncompact type are special Riemannian manifolds which serve as classifying spaces for (torsion-free) arithmetic groups.  As such, their geometry has been studied intensely from several different perspectives.  By a well known theorem of Baily and Borel, if such a manifold admits a parallel complex structure, then it is a complex quasi-projective variety.  In this paper, we study more general indefinite orthogonal groups, which act on Hodge structures of even weight, and draw some conclusions about holomorphic curve counts.\\\\
Let $\Lambda$ be an integral lattice inside $\R^{2,l}$, and $\Gamma$ a congruence subgroup of $O(\Lambda)$. The Baily-Borel Theorem \cite{bb} implies that the double quotient $\Gamma\bs O(2,l)/O(2)\times O(l)$ is a quasi-projective variety. These Hermitian symmetric examples have played a central role in classical moduli theory.  For instance, moduli spaces of polarized K3 surfaces, cubic fourfolds, and holomorphic symplectic varieties are all contained within these Baily-Borel varieties as Zariski open subsets.  Automorphic forms provide natural compactifications for these moduli spaces and bounds on their cohomology.\\\\
One can interpret $O(2,l)/O(2)\times O(l)$ as the set of 2-planes in $\R^{2,l}$ on which the pairing is positive definite.  The presence of the integral lattice $\Lambda$ allows us to define a sequence of $\R$-codimension 2 submanifolds, indexed by $n\in \Q_{>0}$ and $\alpha\in \Lambda^\vee/\Lambda$, where $\Lambda^\vee$ is the dual lattice of covectors taking integral values on $\Lambda$.
$$C_{n,\alpha} : = \Gamma \bs \left(\bigcup_{\substack{v\in \Lambda^\vee,\, v+\Lambda = \alpha\\ (v,v)=-n}} v^\perp\right) \subset \Gamma\bs O(2,l)/O(2)\times O(l).$$
Borel showed \cite{borelalg} that there are finitely many $\Gamma$-orbits of lattice vectors with fixed norm, so the above union is finite in the quotient space.  Each $C_{n,\alpha}$ is isomorphic to a locally symmetric space for $O(2,l-1)$, so it is actually an algebraic subvariety of $\C$-codimension 1 called a {\it Heegner divisor}.  The classes of these divisors in the Picard group satisfy non-trivial relations from the Howe theta correspondence between orthogonal and symplectic groups:
\begin{thm}\label{borch} \cite{borch2}  The formal $q$-series with coefficients in 
$$\Pic_\Q(\Gamma \bs O(2,l)/O(2)\times O(l))\otimes \Q[\Lambda^\vee/\Lambda]$$ 
given by
$$e(V^\vee)e_0+\sum_{n,\alpha} [C_{2n,\alpha}] e_\alpha q^{n}$$
transforms like a $\Q[\Lambda^\vee/\Lambda]$-valued modular form with respect to the Weil representation of the metaplectic group $Mp_2(\Z)$.  Here $\{e_\alpha\}$ denotes the standard basis for $\Q[\Lambda^\vee/\Lambda]$, and $e(V^\vee)$ is the Euler class of the (dual) tautological bundle of positive definite 2-planes.
\end{thm}\noindent
In other words, the $q$-series above lies in the finite dimensional subspace
$$\Mod\left(1+\frac{l}{2},Mp_2(\Z), \Q[\Lambda^\vee/\Lambda] \right)\otimes \Pic_\Q(\Gamma \bs O(2,l)/O(2)\times O(l)).$$
Theorem \ref{borch} has been used to describe the Picard group of moduli spaces \cite{mukai}, and also has applications to enumerative geometry, initiated by \cite{mp}.  In this paper, we move beyond the Hermitian symmetric space to more general symmetric spaces of orthogonal type.  These no longer have a complex structure, and their arithmetic quotients are no longer algebraic, but they still have a theta correspondence, and thus an analogous modularity statement for special cycles in singular cohomology:
\begin{thm}\label{km}\cite{km}  Assume for convenience of this exposition that $\Lambda\subset \R^{p,l}$ is even and unimodular.  Then the formal $q$-series
$$e(V^\vee)+\sum_{n\geq 1} [C_{2n}] q^n \in \Q[[q]]\otimes_\Q H^p(\Gamma \bs O(p,l)/O(p)\times O(l),\Q)$$
lies in the finite-dimensional subspace of modular forms:
$$\Mod\left(\frac{p+l}{2},SL_2(\Z)\right)\otimes H^p(\Gamma \bs O(p,l)/O(p)\times O(l),\Q).$$
Here $e(V^\vee)$ is the Euler class of the dual tautological bundle of p-planes. See Remark \ref{just} for a justification of the level group $SL_2(\Z)$.
\end{thm}\noindent
Theorem \ref{km} will be used to enumerate smooth rational curves on certain elliptically fibered varieties $X\to Y$.  We give a general formula which applies to Weierstrass fibrations over hypersurfaces in projective space.  The answers are honest counts, not virtual integrals, and are expressed in terms of $q$-expansions of modular forms.\\\\
All period domains $D$ for smooth projective surfaces with positive geometric genus admit smooth proper fibrations
$$D \to O(p,l)/O(p)\times O(l).$$  
The Noether-Lefschetz loci in $D$ are the pre-images of special sub-symmetric spaces of $\R$-codimension $p$.  This provides a valuable link between moduli theory and the cohomology of locally symmetric spaces.  We expect the ideas developed in this paper to compute algebraic curve counts on a broad class of varieties. The results are consistent with general conjectures in\cite{oberpix} for elliptic fibrations. \\\\
Let $Y\subset \P^{m+1}$ be a smooth hypersurface of degree $d$ and dimension $m\geq 2$.  For an ample line bundle $\L= \O_Y(k)$, a Weierstrass fibration over $Y$ is a hypersurface
$$X\subset \P(\L^{\otimes -2}\oplus \L^{\otimes -3} \oplus \O_Y)$$
cut out by a global Weierstrass equation (see Section \ref{ellipticsection} for details).  For general choice of coefficients, $X$ is smooth of dimension $m+1$, and the morphism $\pi:X\to Y$ is flat with generic fiber of genus one.  Since $\pi$ admits a section $i:Y\to X$, the generic fiber is actually an {\it elliptic} curve.\\\\
The second homology group of $X$ is given by
$$H_2(X,\Z) \simeq H_2(Y)\oplus \Z f = \Z \ell +\Z f,$$
where $\ell$ is the line class on $Y$ pushed forward via $i$, and $f$ is the class of a fiber.  We begin by posing the naive:
\begin{question}\label{enquest}
How many smooth rational curves are there on $X$ in the homology class $\ell+ nf$?
\end{question}\noindent
The deformation theory of curves $C\subset X$ allows us to estimate when this question has a finite answer.  The expected dimension of the moduli space of curves in $X$ is given by the Hirzebruch-Riemann-Roch formula:
$$h^0(C,N_{C/X}) - h^1(C, N_{C/X}) = \int_C c_1(T_X) + (1-g)(\dim X - 3).$$
The adjunction formula gives $-c_1(T_X)= K_X = \pi^*(K_Y + c_1(\L))$.
\begin{remark} Since $K_X$ is pulled back from $Y$, we have $K_X\cdot f=0$, so our dimension estimate is independent of $n$.  This feature holds more generally for any morphism $\pi$ with $K$-trivial fibers. \end{remark}\noindent
We expect a finite answer to Question \ref{enquest} whenever
\begin{align*}
0 &=-K_X\cdot (\ell + nf) + (m-2) \iff\\
k &= 2m-d.
\end{align*}
Recall that $\L = \O_Y(k)$ was the ample line bundle used to construct the Weierstrass fibration $X\to Y$, so for the rest of the paper, we require that $k=2m-d>0$.  Note that $X$ is Calabi-Yau if and only if $\dim(X)=3$.  Our main result is the
\begin{thm}\label{main}
A very general Weierstrass model $X\to Y$ constructed using $\L = \O_Y(k)$ contains finitely many smooth rational curves in the class $\ell+nf$, whose count we denote $r_X(n)$.  For $k\leq 4$, the generating series is given by
$$\sum_{n\geq 1} r_X(n) q^n = \varphi(q) - \Theta(q),$$
where $\varphi(q)\in \Mod(6k-2,SL_2(\Z))$, and $\Theta(q)\in \Q[\theta_{A_1},\theta_{A_2}, \theta_{A_3}]_{<k}$, a polynomial of weighted degree $< k$.\end{thm}\noindent
Recall that for a lattice $A$, the associated theta series is given by
$$\theta_{A}(q) = \sum_{v\in A} q^{(v,v)/2},$$
and we assign the weight $\rho$ to the series $\theta_{A_\rho}$ for the root lattice $A_\rho$.\\\\
In short, the curve counts $r_X(n)$ are controlled by a finite amount of data, since $\Mod(6k-2,SL_2(\Z))$ and $\Q[\theta_{A_1}, \theta_{A_2},\dots]_{<k}$ are finite dimensional $\Q$-vector spaces.  The series can be explicitly computed when $k\leq 3$.  We record some numerical examples in Section \ref{count} to illustrate the scope of Theorem \ref{main}.
\begin{remark}
We expect that Theorem \ref{main} can be extended to $k \leq 8$, but the statement is less tidy and involves the root systems $D_4$, $E_6$, and $E_7$. For $k>8$, the elliptic surfaces involved will have singularities worse than ADE, so more sophisticated techniques are needed.
\end{remark}\noindent
The argument proceeds roughly as follows.  The curves $C$ that we wish to count have the property that $\pi(C)\subset Y$ is a line.  In other words, they can be viewed as sections of the elliptic fibration
$$\pi^{-1}(\pi(C)) \to \pi(C) \simeq \P^1.$$
$$\includegraphics[width=50mm]{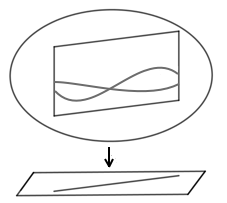}$$
As the line $L\subset Y$ varies, this construction produces a family of elliptic surfaces over the Fano variety of lines in $Y$:
$$\nu:\mathscr S \to F(Y).$$
We wish to count points $[L]\in F(Y)$ such that $\mathscr S_{[L]} = \pi^{-1}(L)$ contains a section curve other than the identity section, i.e. a non-trivial Mordell-Weil group.  The Shioda-Tate sequence expresses the Mordell-Weil group of an elliptic surface in terms of its N\'{e}ron-Severi lattice and the sublattice $V(S)$ spanned by vertical classes and the identity section:
\begin{equation}\label{stses}
0\to V(S) \to \NS(S) \to \MW(S/\P^1) \to 0.
\end{equation}
The period domain for a given class of elliptic surfaces is related to a locally symmetric space, whence the modular form $\varphi(q)$, which counts surfaces with jumping Picard rank.  To obtain the counts $r_X(n)$ we subtract contributions $\Theta(q)$ from surfaces with jumping $V(S)$, which are precisely those with $A_\rho$ singularities.  The terms in the difference formula of Theorem \ref{main} are matched to the groups in the short exact sequence (\ref{stses}).\\\\
The paper is organized as follows.  In Section \ref{ellipticsection}, we review the basic theory of elliptic fibrations and set up the tools for proving transversality of intersections in moduli.  In Section \ref{periods}, we review the theory of period domains and lattices in the cohomology of elliptic surfaces.  Section \ref{simulres} is devoted to the deformation and resolution of $A_\rho$ singularities in families of surfaces, and we introduce the monodromy stack of such a family.  Section \ref{modstatement} explains how Noether-Lefschetz intersection numbers on the period stack satisfy a modularity statement from Theorem \ref{km}.  In Section \ref{discsection}, we use the fact that $k\leq 4$ to classify the singularities which occur in the family $\nu$ at various codimensions, and compute their degrees in terms of Schubert intersections.  Finally, Section \ref{count} explains how to compute the modular form $\varphi(q)$ when $k\leq 3$, and the general form of the correction term $\Theta(q)$.\\\\
{\bf Ackowledgments.}  The author is deeply indebted to his advisor Jun Li for guidance throughout the project, and to Jim Bryan for providing the initial question.  He has also benefitted from conversations with Philip Engel, Tony Feng, Zhiyuan Li, Davesh Maulik, Georg Oberdieck, Arnav Tripathy, Ravi Vakil, and Abigail Ward.

\section{Elliptic Fibrations}\label{ellipticsection}\noindent
We begin by reviewing the Weierstrass equation for elliptic curves in $\P^2$:
\begin{equation}\label{weier}
y^2z = x^3 + Axz^2 + Bz^3.
\end{equation}
This cubic curve has a flex point at $[0:1:0]$, which serves as the identity of a group law in the smooth case.   The curve is singular if and only if the right hand side has a multiple root, which occurs when $\Delta=4A^3+27B^2=0$.  These singular curves are all isomorphic to the nodal cubic, except for when $A=B=0$, which corresponds to the cuspidal cubic.\\\\
To replicate this construction in the relative case, let $Y$ be a smooth projective variety, and $\L\in \Pic(Y)$ an ample line bundle.  We form the $\P^2$ bundle
$$\P(\L^{\otimes -2} \oplus \L^{\otimes -3} \oplus \O_Y) \to Y.$$
The same Weierstrass equation (\ref{weier}) makes sense
for $x,y,z$ fiber coordinates, and
$$ A \in H^0(Y,\L^{\otimes 4}), \quad B \in H^0(Y,\L^{\otimes 6}).$$
Let $X\subset \P(\L^{\otimes -2}\oplus \L^{\otimes -3} \oplus \O)$ be the solution of the global Weierstrass equation and $\pi: X\to Y$ the morphism to the base.  The fibers of $\pi$ are elliptic curves in Weierstrass form, and there is a global section $i:Y\to X$ given in coordinates by $[0:1:0]$, which induces a group law on each smooth fiber.  Now,
$$\Delta = 4A^3+27B^2\in H^0(Y,\L^{\otimes 12})$$
cuts out a hypersurface in $Y$ whose generic fiber is a nodal cubic.  The singularities of $\Delta$ occur along the smooth complete intersection $(A)\cap (B)$, and are analytically locally isomorphic to
$$(\mathrm{cusp}) \times \C^{m-2}.$$ 
The case of $m=2$ and $d=1$ is pictured below, with three fibers drawn over points in $\P^2$ in different singularity strata of $\Delta$.
$$\includegraphics[width=50mm]{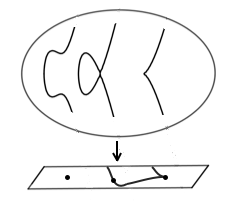}$$
By the adjunction formula applied to $X\subset \P(\L^{\otimes -2}\oplus \L^{\otimes -3}\oplus \O_Y)$,
\begin{align*}
K_X &= (K_{\P(\L^{\otimes -2}\oplus \L^{\otimes -3}\oplus \O)} + [X] )|_X\\
 &= (\pi^*K_Y - 5\pi^* c_1(\L) - 3\zeta) + (3\zeta + 6\pi^* c_1(\L))\\
 &= \pi^*(K_Y + c_1(\L)),
\end{align*}
so the relative dualizing sheaf $\omega_{X/Y}$ is isomorphic to $\pi^*\L$.  By the adjunction formula applied to $Y\subset X$ through the section $i$,
\begin{align*}
K_Y &= (K_X + [Y])|_Y \\ 
 &= K_Y + c_1(\L) + c_1(N_{Y/X}),
\end{align*}
so the normal bundle $N_{Y/X}$ is isomorphic to $\L^\vee$.
\begin{defn}
The parameter space for Weierstrass fibrations over $Y$ is given by the weighted projective space
$$W(Y,\L) := \left(H^0(Y,\L^{\otimes 4}) \oplus H^0(Y,\L^{\otimes 6})-\{0\}\right)/\C^\times,$$
where $\C^\times$ acts on the direct summands with weight $2$ and $3$, respectively.
\end{defn}\noindent
In this paper, $X$ is always a general member of $W(Y,\O(k))$, where $Y\subset \P^{m+1}$ is a smooth hypersurface of degree $d$, and $k=2m-d$.  Since we are interested in counting curves in $X$ which lie over lines in $Y$, we will also consider elliptic surfaces $S\in W(\P^1,\O(k))$.
For convenience, we gather some properties of these surfaces.
\begin{prop}
For $S\in W(\P^1,\O(k))$ a smooth surface, its Hodge numbers are:
\begin{align*}
h^1(S,\O_S)&=0;\\
h^2(S,\O_S)&=k-1;\\
h^1(S,\Omega_S)&=10k.
\end{align*}
As a result, $e(S) = 12k$, which is the number of singular fibers, and the canonical line bundle is given by
$$\omega_S \simeq \pi^* \O_{\P^1}(k-2).$$
\end{prop}
\begin{proof}
These are standard computations using Noether's formula and the Leray spectral sequence for $\pi$; see Lecture III of \cite{mir}.
\end{proof}\noindent
\begin{thm}
Any elliptic surface $S\to \P^1$ is birational to a Weierstrass surface.  Furthermore, there is a bijection between (isomorphism classes of) smooth relatively minimal surfaces and Weierstrass fibrations with rational double points.
\end{thm}
\begin{proof}
This uses Kodaira's classification of singular fibers; see Lecture II of \cite{mir}.
\end{proof}\noindent
It will be convenient for us to take a further quotient of $W(\P^1,\O(k))$ to account for changes of coordinates on the base.
\begin{defn}
The moduli space for Weierstass surfaces over $\P^1$ is given by the (stack) quotient
$$\W_k := W(\P^1,\O(k)) / PGL(2).$$
\end{defn}
\begin{remark}
Miranda showed in \cite{mirmoduli} that $[S]\in W(\P^1,\O(k))$ is GIT stable with respect to $PGL(2)$ if and only if it has rational double points, so $\W_k$ has a quasi-projective coarse space with good modular properties.
\end{remark}\noindent
For any variety $Y\subset \P^{m+1}$, the locus of lines contained in $Y$ is called the Fano scheme\footnote{There is a natural scheme structure on $F(Y)$ coming from its defining equations in the Grassmannian, but for our purposes $Y$ is general, so $F(Y)$ is a variety.} of $Y$, and is denoted
$$F(Y) \subset \G(1,m+1).$$
\begin{thm}\label{fano}
For a general hypersurface $Y\subset \P^{m+1}$ of degree $d$, the Fano scheme is smooth of dimension $2m-d-1 = k-1$, for $k>0$.
\end{thm}
\begin{proof}
To study the general behavior, we construct an incidence correspondence
$$\Omega = \{ (L,Y): L\subset Y \}\subset \G(1,m+1) \times \P^N.$$
The first projection $\Omega\to \G(1,m+1)$ is surjective and has linear fibers of dimension $N-d-1$, so $\Omega$ is smooth and irreducible of dimension $N+2m-d-1$.  The second projection has fiber $F(Y)$ over $[Y]\in \P^N$.  To get the desired dimension, it suffices to show that a general hypersurface $Y$ contains a line, so that the second projection $\Omega \to \P^N$ is surjective.  This can be done by constructing a smooth hypersurface containing a line whose normal bundle is balanced as in \cite{he}.
\end{proof}\noindent
Moreover, if we vary the hypersurface $Y$, then $F(Y)$ varies freely inside $\G(1,m+1)$.  To be precise, 
\begin{defn}\label{freely}
Let $Z\to B$ be a submersion of complex manifolds, and let $f:Z\to P$ be a family of immersions $\{f_b:Z_b \to P\}$.  This deformation is called {\it freely movable} if for any $x_0\in Z_0$ and any $v\in T_{f(x_0)}P$, there exists a 1-parameter subfamily $T\subset B$ and a section $x(t)\in Z_t$ such that $x(0)=x_0$ and $\frac{d}{dt}|_{t=0}(f\circ x)=v$.
\end{defn}\noindent
The second projection $\Omega \to \P^N$ from Theorem \ref{fano} is generically smooth, and the family of embeddings $\Omega \to \G(1,m+1)$  is freely movable because for any line $L\in F(Y_0)$ and tangent vector $v\in T_{[L]}\G$, there is a curve of lines $\{L_t\}$ in the direction $v$.  Since every line lies on a hypersurface, there is a deformation $Y_t$ such that $[L_t]\in F(Y_t)$.  This property is useful for proving transversality statements, using the
\begin{lemma}\label{bertini}
Let $(Z\to B, f:Z \to P,)$ be freely movable, and fix some subvariety $\Pi\subset P$.  Then for general $b\in B$, $Z_b$ intersects $\Pi$ transversely.
\end{lemma}
\begin{proof}
We argue using local holomorphic coordinates.  Suppose that $\Pi\subset P$ is simply $\C^\ell\subset \C^n$, and that $f_b:D^r\to P$ is an embedding.  Assume for the sake of contradiction that the locus $\Sigma \subset D^r\times B$ where $f_b$ is not transverse to $\Pi$ surjects onto $B$.  We may choose $0\in B$ such that $\Sigma\to B$ does not have a multiple fiber over $0$.  If $f_0$ is non-transverse to $\Pi$ at $p\in D^r$, then we have 
$$\C^\ell + df_0(T_pD^r) \subsetneq \C^n.$$
Let $\vec{v}$ be a vector outside the subspace above, and use the hypothesis of free movability to find a subfamily $f_t:D^r\to P$ and section $x(t)\in D^r$ such that $x(0)=p$ and $\frac{d}{dt}|_{t=0}(f\circ x)=v$.  Since $\Sigma$ does not have a multiple fiber, there exists another section $y(t)\in D^r$ such that $y(0)=p$ and $f_t(y(t))$ meets $\Pi$ non-transversely.  Now
$$f_t(y(t))- f_0(p) = \left( f_t(y(t))- f_t(x(t)) \right) + \left( f_t(x(t)) - f_0(p)   \right).$$
At first order in $t$, the left hand side lies in $\C^\ell$, the first term on the right hand side lies in the image of $df_0$, and the second term on the right hand side lies in the span of $\vec{v}$.  This contradicts our choice of $v$.  Since transversality is Zariski open, we obtain the statement for general $b\in B$.\end{proof}
\noindent
Any smooth curve $C\subset X$ with class $\ell + nf$ maps isomorphically to a line $L\subset Y$.  The pre-image $\pi^{-1}(L)$ will contain $C$, so to set up the enumerative problem, we consider the family of all elliptic surfaces over lines in $Y$.  
\begin{defn}
Let $U\to F(Y)\times Y$ be the universal line, and form the fibered product
$$\mathscr S := X\times_Y U.$$
\end{defn}\noindent
The natural morphism $\nu:\mathscr S \to F(Y)$ is flat by base change and composition:
$$\xymatrix{
  & \mathscr S \ar[d]^{\pi'} \ar[ld]_{\nu} \ar[r] & X \ar[d]^{\pi} \\
F(Y) & U \ar[l] \ar[r] & Y.
}$$
The family $\nu$ will be our primary object of study.  Its fiber over a line $[L]\in F(Y)$ is simply $\pi^{-1}(L)$.  Proposition \ref{adeonly} shows that for $k\leq 4$, the fibers of $\nu$ have no worse than isolated $A_\rho$ singularities.  We have an associated moduli map to the Weierstrass moduli space
$$\mu_X:F(Y) \to \W_k$$
by restricting the global Weierstrass equation from $Y$ to $L$.  
\begin{lemma}\label{immersion}
The map $\mu_X$ is an immersion for general $Y$ and $X\in W(Y,\O(k))$.
\end{lemma}
\begin{proof}
This is a statement about unordered point configurations on $\P^1$.  The argument is rather technical and is relegated to Appendix B.
\end{proof}\noindent
If we fix $Y$ and vary $X\in W(Y,\O(k))$, we obtain a family of immersions $\mu_X$.
\begin{prop}
The family of immersions given by $F(Y)\times W(Y,\O(k)) \to \W_k$ is freely movable in the sense of Definition \ref{freely}.
\end{prop}
\begin{proof}
This follows from surjectivity of the restriction map
$$H^0(Y,\O(4k))\oplus H^0(Y,\O(6k)) \to H^0(L,\O(4k))\oplus H^0(L,\O(6k)).$$
There is no need to vary the line $L$, only $[A:B]\in W(Y,\O(k))$.
\end{proof}\noindent
By Lemma \ref{bertini}, we may assume after deformation that $\mu_X$ is transverse to any fixed subvariety of $\W_k$. This will be applied to the Noether-Lefschetz loci inside $\W_k$. The intersection of $\mu_X$ with the discriminant divisor in $\W_k$ is responsible for the correction term $\Theta(q)$ in Theorem \ref{main}.

\section{Noether-Lefschetz Theory}\label{periods}\noindent
Let $S\in \W_k$ be an elliptic surface.  Its Picard rank is automatically $\geq 2$ because its N\'{e}ron-Severi group $\NS(S)$ contains the fiber class $f$ and the class of the identity section $z$.  Any section class has self-intersection $-k$ by adjunction, so $\NS(S)$ contains the rank 2 lattice:
$$\langle f,z\rangle = \begin{pmatrix} 0 & 1 \\ 1 & -k \end{pmatrix}.$$
We refer to this sublattice as the polarization, and it comes naturally from the elliptic fibration structure.
\begin{remark} \label{canfib} Except for the K3 case ($k=2$), the elliptic fibration structure on $S$ is canonical because $K_S$ is a nonzero multiple of the fiber class.\end{remark}\noindent
\begin{defn}
An elliptic surface $S\in \W_k$ is called Noether-Lefschetz special if its relatively minimal resolution has Picard rank $>2$.
\end{defn}
\begin{thm}\cite{cox}\label{cox}
All components of the Noether-Lefschetz locus in $\W_k$ are reduced of codimension $k-1$, except for the discriminant divisor, which is codimension 1.
\end{thm}\noindent
Noether-Lefschetz theory is the study of Picard rank jumping in families of surfaces.  The short exact sequence of Shioda-Tate for an elliptic surface clarifies the two potential sources of jumping:
\begin{equation}\label{shiodatate}
0\to V(S)\to \NS(S) \to \MW(S/\P^1) \to 0.
\end{equation}
Here $V(S)$ is the sublattice spanned by the zero section class and all vertical classes, and $\MW(S/\P^1)$ is the Mordell-Weil group of the generic fiber, which is an elliptic curve over $\C(\P^1)$.  If $S$ is a smooth Weierstrass fibration, then all fibers are integral, so $V(S)$ is simply the polarization sublattice.  The group $\MW(S/\P^1)$ will be torsion-free in the cases that concern us (see Appendix A), so the intersection form on $\NS(S)$ splits the short exact sequence (\ref{shiodatate}).  In particular, the orthogonal projection $\Pi: \NS(S) \to V(S)^\perp$ induces an isomorphism of groups
$$\MW(S/\P^1) \to V(S)^{\perp}.$$
\begin{lemma}\label{shift1}
If $\sigma\in \NS(S)$ is the class of a section curve, then its orthogonal projection to $V(S)^\perp$ has self-intersection
$$-2(z\cdot \sigma+k).$$
\end{lemma}
\begin{proof}
Since $\MW(S/\P^1)$ is torsion-free, $\sigma$ is orthogonal to all the exceptional curves.  The projection can be computed by applying Gram-Schmidt to the polarization sublattice $\langle f,z\rangle$.
\end{proof}
\begin{lemma}\label{gplaw} If $\sigma$ is a section curve, and $\sigma^{*m}$ its $m$-th power with respect to Mordell-Weil group law, then the class of $\sigma^{*m}$ in $\NS(S)$ is given by
$$ \left[ \sigma^{*m} \right] = m\sigma - (m-1)z + (z\cdot \sigma+k)m(m-1) f. $$
In particular $(\sigma^{*m})\cdot z$ grows quadratically with $m$.
\end{lemma}
\begin{proof}
The class on the generic fiber is computed using the Abel-Jacobi map for elliptic curves.  To determine the coefficient of $f$, use the fact that any section curve has self-intersection $-k$ in $\NS(S)$.
\end{proof}
\begin{lemma}\label{shift2} 
If $\sigma$ is the class of a section curve, and $\iota:S \to X$ is the inclusion morphism, then
$$ \iota_*(\sigma) = \ell+(z\cdot\sigma + k )f \in H_2(X,\Z) $$
\end{lemma}
\begin{proof}
The class can be computed by intersecting with complementary divisors.  The global section $i(Y)\subset X$ has normal bundle $\O_Y(-k)$, whence the shift.
\end{proof}\noindent
Setting $\NS_0(S):=\langle f,z\rangle^\perp \subset \NS(S)$ and $V_0(S) := \NS_0(S)\cap V(S)$, the sequence
$$0\to V_0(S) \to \NS_0(S) \to \MW(S/\P^1)\to 0$$
is also split exact.  The lattice $V_0(S)$ will be a root lattice spanned by the classes of exceptional curves.\\\\
To set up the Noether-Lefschetz jumping phenomenon, we consider the full polarized cohomology lattice
$$\Lambda(S):= \langle f,z\rangle^\perp \subset H^2(S,\Z).$$
\begin{thm}
As abstract lattices, $\Lambda(S)\simeq H^{\oplus {2k-2}}\oplus E_8(-1)^{\oplus k}$, where $H$ denotes the rank 2 hyperbolic lattice, and $E_8(-1)$ denotes the $E_8$ lattice with signs reversed.
\end{thm}
\begin{proof}
By Poincar\'{e} duality, the pairing on $H^2(S,\Z)$ is unimodular, and the Hodge Index Theorem gives its signature to be $(2k-1,10k-1)$.  The polarization sublattice $\langle f,z\rangle$ is unimodular, so its orthogonal complement is as well.  The Wu formula for Stiefel-Whitney classes reads
$$\alpha\cdot \alpha \equiv \alpha\cdot K_S \,\, \text{(mod 2)},$$
so $\alpha\cdot\alpha\in 2\Z$ for $\alpha\in \langle z,f\rangle^\perp$.  By the classification of indefinite unimodular lattices, there is a unique even lattice of signature $(2k-2,10k-2)$, namely the one above.
\end{proof}
\noindent By the Lefschetz $(1,1)$ Theorem, we have
$$\NS_0(S)\simeq (H^{2,0}(S)\oplus H^{0,2}(S))_\R^\perp \cap \Lambda(S),$$
so the Picard rank jumping can be detected from the Hodge structure of $S$.  There is a period domain which parametrizes polarized\footnote{For the rest of the paper, all Hodge structures are assumed to be polarized.} Hodge structures on the abstract lattice $\Lambda$.  A weight 2 Hodge structure can be interpreted as a representation of the Deligne torus $\mathbb S^1 = \text{Res}_{\C/\R}\C^\times$, valued in the orthogonal group $O(\Lambda_\R)$:
$$\psi: \mathbb S^1 \to O(\Lambda_\R),$$
such that $\psi(t)=t^2$ for $t\in \R^\times$.  The Hodge decomposition comes from extending linearly to $\Lambda_\C$, and setting
$$H^{p,q}(\psi) = \{v\in \Lambda_\C:  \psi(z)\cdot v = z^p \ov{z}^q v \}.$$
For fixed Hodge numbers, the group $O(\Lambda_\R)$ acts transitively via conjugation on the set of such representations.  This realizes the relevant period domain as a homogeneous space:
$$D  \simeq O(2k-2,10k-2)/\pm U(k-1)\times O(10k-2).$$
This can alternatively be viewed as an open orbit inside a complex flag variety, so in particular it has a complex structure.  It contains a sequence of Noether-Lefschetz loci, which are given by
$$\widetilde\NL_{2n} := \bigcup_{\beta\in\Lambda,\,(\beta,\beta) =-2n} \{\psi\in D : H^{2,0}(\psi)\subset \beta^\perp \},$$
each of which is simultaneously a homogeneous space 
$$\widetilde\NL_{2n} \simeq O(2k-2,10k-3)/\pm U(k-1)\times O(10k-3)$$ 
and a complex submanifold of $\C$-codimension $k-1$.  These loci parametrize Hodge structures on $\Lambda$ which potentially come from a surface $S$ with $\NS_0(S)\neq 0$, since $\beta\in \NS_0(S)$.\\\\
The family $\mathscr S \to F(Y)$ is generically smooth, so we have a holomorphic period map
$$j: F(Y) \dashrightarrow \Gamma \bs D ,$$
defined away from the singular locus, where $\Gamma$ is the image of monodromy for the smooth family, which lies in the arithmetic group $O(\Lambda)\subset O(\Lambda_\R)$.
\begin{prop}
The period map $j$ is an immersion away from the singular locus.
\end{prop}
\begin{proof}
Combine Lemma \ref{immersion} with the infinitesimal Torelli theorem of M. Saito \cite{saito} for deformations of smooth elliptic surfaces.
\end{proof}\noindent
By Proposition ~\ref{adeonly}, singularities in the fibers of $\nu$ are ADE type when $k\leq 4$, so the local monodromy of the smooth family is finite order.  This allows us to extend $j$ over all of $F(Y)$ on general grounds \cite{schmid}.  The extension can be understood explicitly in terms of a simultaneous resolution (see Theorem \ref{brieskorn}).  The period image of a singular surface is the Hodge structure of its minimal resolution.\\\\
The Noether-Lefschetz numbers of the family $\mathscr S\to F(Y)$ are morally the intersections of $j_*[F(Y)]$ with
$$\NL_{2n} := \Gamma \bs \widetilde{\NL}_{2n} \subset \Gamma \bs D .$$
However, since $\Gamma$ contains torsion elements, the period space $\Gamma \bs D $ has singularities.  To compute the topological intersection product we consider instead the smooth analytic stack quotient $[\Gamma \bs D ]$.  The period map $j$ does not lift to this stack, so in Section \ref{simulres} we construct a stack $\F(Y)$ with coarse space $F(Y)$, admitting a map
$$\mathfrak j:\F(Y) \to [O(\Lambda) \bs D ].$$
lifting the classical period map, which extends to all of $\F(Y)$.\\\\
Lastly, we note that the Noether-Lefschetz loci $\NL_{2n}\subset [O(\Lambda) \bs D]$ are irreducible after fixing the divisibility of the lattice vector. This follows from a uniqueness theorem for embeddings of rank 1 lattices into a unimodular lattice; see Theorem 1.1.2 of \cite{nik}. The locus $\NL_{2n}\subset [O(\Lambda) \bs D]$ decomposes into components indexed by $m\in \mathbb N$ such that $m^2|n$.  Let $v_m\in \Lambda$ be a lattice vector of self-intersection $2n/m^2$ so that $mv_m$ has self-intersection $2n$.  Then we can write
$$\NL_{2n} = \bigcup_{m^2|n} O(\Lambda)\bs \{\psi\in D : H^{2,0}(\psi)\subset v_m^\perp\}.$$

\section{Simultaneous Resolution}\label{simulres}\noindent
In this section, we study flat families of surfaces with rational double points, focusing on the $A_\rho$ case.
\begin{thm}\label{brieskorn}\cite{brieskornsimul}
Let $\pi:X \to B$ be a flat family of surfaces over a smooth variety $B$, such that each fiber $ X_b$ has at worst ADE singularities.  Then after a finite base change $B'\to B$ in the category of analytic spaces, the new family $\pi':X'\to B'$ admits a {\it simultaneous resolution}, a proper birational morphism $\widetilde{ X}'\to X'$ which restricts to a minimal resolution on each fiber of $\pi$.
\end{thm}\noindent
\'{E}tale locally it suffices to consider a versal family.  In the case of the $A_\rho$ singularity $(x^2+y^2=z^{\rho+1})$, this family is given by
\begin{equation}\label{versala}
x^2+y^2 = z^{\rho+1}+ s_{1} z^{\rho-1} + s_{2}z^{\rho-2}+ \dots + s_{\rho}
\end{equation}
in the deformation coordinate $\vec{s}\in S=\C^\rho$.  The base change required is given by the elementary symmetric polynomials
$$s_i = \sigma_{i+1}(\vec{t}),$$
where $\vec{t}\in T=\Spec \C[t_0,t_1,\dots,t_\rho]/\sum t_i \simeq \C^\rho$.  The cover is Galois with deck group $\mathfrak S_{\rho+1}$, branched over the discriminant hypersurface.  After base change, the equation can be factored
$$x^2 + y^2 = \prod_{i=0}^\rho (z+t_i).$$
Singularities in the fibers occur over the big diagonal in $T$, which is the hyperplane arrangement dual to the root system $A_\rho\simeq \Z^\rho\subset \C^n$.  The $\mathfrak S_3$ covering $T\to S$ in the case of $A_2$ is pictured below.
$$\includegraphics[width=50mm]{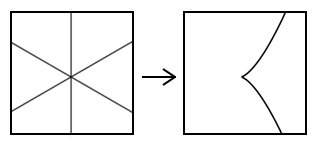}$$
A simultaneous resolution can be obtained by taking a small resolution of the total space.  This is far from unique; one can write the total space as an affine toric variety, and then choose any simplicial subdivision of the single cone:
$$\Z_{+}\langle \vec{e}_1, \vec{e}_2,\dots, \vec{e}_{\rho+2}, \vec{e}_1-\vec{e}_2+\vec{e}_3, \vec{e}_1-\vec{e}_2+\vec{e}_4,\dots, \vec{e}_1-\vec{e}_2+e_{\rho+2} \rangle \subset \Z^{\rho+2}.$$
The smooth fiber in both families is diffeomorphic to a resolved $A_\rho$ singularity, and thus has cohomology lattice
$$H^2(X_s,\Z) \simeq A_\rho(-1).$$
The Gauss-Manin local system on $S-\Delta$ corresponds to the representation
$$\pi_1(S-\Delta) \simeq Br_{\rho+1} \to \mathfrak S_{\rho+1} \to O(A_\rho),$$
where $Br_{\rho+1}$ denotes the braid group, and $\mathfrak S_{\rho+1}$ is the Weyl group of $A_\rho$ acting by reflections.  The base change morphism $T\to S$ can be interpreted as the quotient $\C^\rho\to \C^\rho/\mathfrak S_{\rho+1}\simeq \C^\rho$ by extending the Weyl group action $\C$-linearly to $A_\rho\otimes \C$.  The quotient map is ramified along the dual hyperplane arrangement and branched over the discriminant $\Delta$.\\\\
Artin framed Theorem \ref{brieskorn} in the language of representable functors.
\begin{thm}\label{artin}\cite{artinsimul}
For $X\to B$ as above, let $\Res_{ X/B}$ be the functor from $Sch/B \to Set$ which sends
$$[B'\to B] \mapsto \{ \text{simultaneous resolutions }\widetilde{ X}'\to  X' =  X \times_B B' \}.$$
Then $\Res_{X/B}$ is represented by a locally quasi-separated algebraic space.
\end{thm}\noindent
The space is often not separated, even in the case of an ordinary double point $A_1$:
\begin{ex}
Let $X \to \C$ be the versal deformation of $A_1$:  $x^2+y^2+z^2 = t$.  To build a simultaneous resolution, we base change by $t\mapsto t^2$, and then take a small resolution of the threefold singularity $x^2+y^2+z^2 = t^2$.  There are two choices of small resolution, differing by the Atiyah flop.  Hence, the algebraic space $\Res_{ X/\C}$ is isomorphic to $\A^1$, with the \'{e}tale equivalence relation 
$$\mathscr R = \Delta\cup \{ (x,-x):x\neq 0 \} \subset \A^1\times \A^1.$$
\end{ex}\noindent
To do intersection theory, we want a nicer base for the simultaneous resolution, namely a smooth Deligne-Mumford stack.  From the perspective of periods we only need a resolution at the level of cohomology lattices.  Let $\Delta\subset B$ be the discriminant locus of the family, and $j:U\hookrightarrow B$ its complement.  Assuming that $U$ is nonempty, we have a Gauss-Manin local system $R^2\pi_{U*}(\Z)$ on $U(\C)$ whose stalk at $b\in U$ is $H^2(X_b,\Z)$ equipped with the cup product pairing.  
\begin{remark}
In our situation, we will consider instead the primitive cohomology $R^2\pi_{U*}(\Z)_{prim}\subset R^2\pi_{U*}(\Z)$ whose stalk at $b\in U$ is isomorphic to the orthogonal complement of the polarization sublattice $\langle f,z \rangle$ as defined in Section \ref{periods}. This is well-defined over $U$ because $\langle f,z\rangle$ is monodromy invariant.
\end{remark}\noindent
The pushforward sheaf
$$\H := j_* R^2\pi_{U*}(\Z)_{\text{prim}}$$
is a constructible sheaf on $B(\C)$ whose stalk at $b\in \Delta(\C)$ consists of classes invariant under the local monodromy action.
\begin{defn}
Let $\Lambda$ be the stalk of $\H$ over a smooth point $b\in U(\C)$.  A cohomological simultaneous resolution is an embedding $\H \hookrightarrow \LL$ into a local system $\LL$ on $B(\C)$ with stalk $\Lambda$.  
\end{defn}\noindent
If $X\to B$ admits a simultaneous resolution, then the local monodromy action is trivial so $\H$ is a local system already.  If it does not, then the cohomological resolutions are representable by a stack:
\begin{defn}
The monodromy stack $\mathfrak B$ over $B$ is the following category fibered in groupoids.  An object of $\mathfrak B$ is given by a pair
$$(f:B'\to B, i':f^* \H \hookrightarrow \LL'),$$
where $\LL'$ is a local system on $B'(\C)$ with stalk $\Lambda$.  A morphism from $(B',f,\LL',i')$ to $(B'',g,\LL'',i'')$ is a map $h:B'\to B''$ such that $f=g\circ h$, and an isomorphism $\phi:h^*\LL \to \LL$ such that $i'=\phi\circ h^*i''$.
\end{defn}\noindent
To check that $\mathfrak B$ is a stack for the \'{e}tale topology on $B$, we must verify that isomorphisms form a sheaf, and that objects satisfy descent.  Both of these follow from the corresponding facts for local systems and the fact that \'{e}tale morphisms of $\C$-schemes induce local isomorphisms on their underlying analytic spaces.  Automorphisms of an object $(B',f,\LL',i')$ are the automorphisms of $\LL'$ which fix $i'(f^*\H)$.  When $B'$ is a point $p$, then $\H_p$ is the space of local invariant cycles, so the automorphism group is generated by reflections in the vanishing cycle classes.
\begin{thm}\label{dmstack}
The stack $\mathfrak B$ is Deligne-Mumford.  
\end{thm}
\begin{proof}
We have natural morphisms
$$\Res_{X/B} \to \mathfrak B \to B,$$
since a simultaneous resolution induces a cohomological one.  Hence, $\mathfrak B$ is an algebraic stack by Theorem \ref{artin}.  The Deligne-Mumford property follows from the fact that simple surface singularities have finite monodromy.
\end{proof}\noindent
\begin{prop}
If $X\to S$ is the versal family of the $A_\rho$ singularity as described in (\ref{versala}), then its monodromy stack is isomorphic to
$$[T/\mathfrak S_{\rho+1}],$$
which has coarse space $S$.
\end{prop}
\begin{proof}
We define an equivalence of categories fibered over $Sch/S$.  An object of $[T/\mathfrak S_{\rho+1}](Y)$ consists of a principal $\mathfrak S_{\rho+1}$-bundle $E\to Y$ with an equivariant map $\widetilde{f}:E\to T$.  Composing with the coarse space map $[T/\mathfrak S_{\rho+1}] \to S$ produces a map $f:Y\to S$.  Let $\LL$ be the sheaf of sections of the associated bundle
$$E\times_{\mathfrak S_{\rho+1}} A_\rho \to Y,$$
which has fiber $A_\rho$.  We can describe $\H$ as the sheaf of sections of
$$T \times_{\mathfrak S_{\rho+1}} A_\rho \to S.$$
To be single-valued in a neighborhood of $s\in S$, a section of $\H$ must send $s$ to the class of $(t,a)$, where $a$ is fixed by all $g\in \text{Stab}_{\mathfrak S_{\rho+1}} (t)$.  Since the action of $\mathfrak S_{\rho+1}$ is generated by reflections in the roots $R(A_\rho)\subset A_\rho$, this condition is equivalent to: 
$$a\in \left( t^\perp \cap R(A_\rho) \right)^\perp.$$
With this in mind, we define
$$\widetilde{F} :=\left\{ (e,a): a\in  \left( \widetilde{f}(e)^\perp \cap R(A_\rho) \right)^\perp\right\}\subset E\times  A_\rho.$$
This is $\mathfrak S_{\rho+1}$-invariant, so it descends to 
$$F\subset E\times_{\mathfrak S_{\rho+1}} A_\rho $$
over $Y$.  The equivariant map $\widetilde{f}:E\to T$ induces a map
\begin{equation} \label{map}
E\times_{\mathfrak S_{\rho+1}} A_\rho \to \left(T \times_{\mathfrak S_{\rho+1}} A_\rho \right) \times_S Y
\end{equation}
If $\mathcal F$ is the sheaf of sections of $F \to Y$, then (\ref{map}) gives an isomorphism
$$\mathcal F \to f^*\H.$$
Morphisms in $[T/\mathfrak S_{\rho+1}]$ are Cartesian diagrams of principal bundles with commuting equivariant maps, which induce morphisms of the above data.  This defines a functor from $[T/\mathfrak S_{\rho+1}]$ to the monodromy stack.  Recall that the action of $\mathfrak S_{\rho+1}$ on $A_\rho$ gives an equivalence from the category of principal $\mathfrak S_{\rho+1}$-bundles to the category of $A_\rho$-local systems.  The data of commuting maps to $T$ corresponds to the coincidence of cohomology subsheaves, so our functor is fully faithful.\\\\
To show essential surjectivity, let $(f:Y\to S, i:f^*\H \hookrightarrow \LL)$ be an object of the monodromy stack.  The bundle $E$ of Weyl chambers in the stalks of $\LL$ is a principal $\mathfrak S_{\rho+1}$-bundle over $Y$.  Let $\Delta'\subset \Delta$ be the smallest singular stratum containing the image of $f$, which corresponds to a partition of $\rho+1$, or equivalently a conjugacy class of subgroup $G\subset \mathfrak S_{\rho+1}$.  The general stalk of $f^*\H$ is isomorphic to the invariant sublattice $(A_\rho)^{G}$.  We can form the associated bundle
$$\ov{E} := E \times_{\mathfrak S_{\rho+1}} \mathfrak S_{\rho+1}/G$$
The fact that $f^*\H$ extends to a local system over $Y$ implies that it is trivial, so $\ov{E}$ is a trivial bundle.  Any choice of lift $Y \to T$ gives rise to an equivariant map from $\ov{E}\to T$.  Lifting this to an equivariant map $E\to T$ is automatic because each $G$-coset maps to the same point of $T$.
\end{proof}\noindent
More generally, if $X\to B$ has only $A_\rho$ singularities in the fibers, then \'{e}tale locally on $B$, we have a morphism 
$$B\to \prod_j S_j$$
to the versal bases of the isolated singularities in the central fiber.  There is an embedding of the associated root lattice $A\subset \Lambda$, which induces an embedding of the Weyl group $W(A)\subset O(\Lambda)$.  The monodromy of the family lies in $W(A)$, which implies that the diagram
$$
\xymatrix{
\mathfrak B \ar[r] \ar[d] & \prod_j [T_j/\mathfrak S_{\rho_j+1}] \ar[d] \\
B \ar[r] & \prod_j S_j
}
$$
of stacks is 2-Cartesian.  In particular, $B$ is the coarse space of $\mathfrak B$.  If the family has maximal variation at each singularity then $\mathfrak B$ is smooth.\\\\
We apply Theorem \ref{dmstack} to the family $\nu:\mathscr S \to F(Y)$ to obtain a Deligne-Mumford stack $\F(Y)$ such that the (primitive) Gauss-Manin system on the smooth locus extends to a local system $\LL$ on all of $\F(Y)$ with stalk $\Lambda$.  Let $\mathfrak E$ be the principal $O(\Lambda)$-bundle on $\F(Y)$ of isomorphisms from $\LL$ to the constant sheaf $\ul{\Lambda}$.  There is an equivariant map from $\mathfrak E \to D $ sending a point in $\mathfrak E$ to the Hodge structure on $\Lambda$ obtained by identifying the stalk of $\LL$ with $\Lambda$ via the isomorphism.  This data gives the desired stacky period map
$$\mathfrak j:\F(Y) \to [O(\Lambda) \bs D ]$$
extending $j:F(Y)\dasharrow O(\Lambda)\bs D$.

\bigskip
\section{Modularity Statement}\label{modstatement}\noindent
The work of Kudla-Millson produces a modularity statement for intersection numbers in a general class of locally symmetric spaces $M$ of orthogonal type.  In this section, we summarize\footnote{We match the notation of \cite{km} for the most part, but all instances of positive (resp. negative) definiteness are switched.} the material in \cite{km}, and adapt it to our situation.\\\\
Let $M$ be the double quotient $\Gamma \bs O(p,l) / K$ of an orthogonal group on the left by a torsion-free arithmetic subgroup preserving an even unimodular lattice $\Lambda\subset \R^{p+l}$, and on the right by a maximal compact subgroup.  This is automatically a manifold, since any torsion-free discrete subgroup acts freely on the compact cosets.  We can interpret $O(p,l)/K$ as an open subset of the real Grassmannian $Gr(p,p+l)$ consisting of those $p$-planes $Z\subset \R^{p+l}$ on which the form is positive definite.  For any negative definite line $\v\subset\R^{p+l}$, set
$$\widetilde{C}_\v := \{ Z\in O(p,l)/K:  Z\subset \v^\perp \}$$
which is $\R$-codimension $p$.  Indeed, the normal bundle to $\widetilde{C}_\v$ has fiber at $Z$ equal to $\Hom(Z,\v)$.  While the image of $\widetilde{C}_\v$ in $M$ may be singular, it can always be resolved if we instead quotient by a finite index normal subgroup of $\Gamma$.  Furthermore, \cite{km} gives a coherent way of orienting the $\widetilde{C}_\v$, so that it makes sense to take their classes in the Borel-Moore homology group $H^{BM}_{pl-p}(M)\simeq H^p(M)$.\\\\
For any positive integer $n$, the action of $\Gamma$ on the lattice vectors in $\Lambda$ of norm $-2n$ has finitely many orbits \cite{borelalg}.  Choose orbit representatives $\{v_1,\dots, v_k\}$, and set
$$\widetilde{C}_{2n} := \bigcup_{i=1}^k \widetilde{C}_{\langle v_i \rangle}.$$
The image of $\widetilde{C}_{2n}$ in the arithmetic quotient $M$ is denoted by $C_{2n}$.  Locally, $\widetilde{C}_n$ is a union of smooth (real) codimension $p$ cycles meeting pairwise transversely, one for each lattice vector of norm $-2n$ orthogonal to $Z$.  We quote the following result directly from \cite{km}, in the case of $Sp(1)\simeq SL(2)$:
\begin{thm}\label{kmnum} 
For any homology class $\alpha\in H_p(M)$, the series
$$ \alpha\cap e(V^\vee) + \sum_{n=1}^\infty  \alpha\cap [C_{2n}] \,  e^{2\pi i n\tau}$$
is a classical modular form for $\tau\in \mathbb H$ of weight $(p+l)/2$.  The constant term is the integral of the Euler class of the (dual) tautological bundle of $p$-planes.\end{thm}
\begin{remark}\label{just}
The statement in \cite{km} does not specify the level, but when the lattice $\Lambda$ is even and unimodular, we will see that the level group is the full modular group $SL_2(\Z)$. The proof of Theorem \ref{kmnum} proceeds by considering the functional $\Theta:\mathscr S(\Lambda\otimes \R) \to \C$ on Schwartz functions given by a sum of delta functions supported at lattice points:
$$\Theta = \sum_{v\in \Lambda} \delta_v.$$
Now $\mathscr S(\Lambda\otimes \R)$ admits an action of $O(\Lambda\otimes \R)\times Mp_2(\R)$, where the first factor acts by pre-composition, and the second factor acts via the Weil representation. Explicitly, the generators $S$ and $T$ of $Mp_2(\Z)$ act via
\begin{align*}
(T\cdot f)(x) &= e^{\pi i (x,x)} f(x);\\
(S\cdot f)(x) &= \frac{e^{\mathrm{sgn}(\Lambda)\pi i/4}}{\sqrt{\det(\Lambda^\vee)}}\widehat{f}(x),
\end{align*}
where $\widehat{f}$ denotes the Fourier transform. Since $\Gamma\subset O(\Lambda)$, we see that $\Theta$ is $\Gamma$-invariant. Observe that $\Theta$ is also $Mp_2(\Z)$-invariant, using the Poisson summation formula combined with the fact that $\Lambda$ is even (for $T$), unimodular (for $S$), and $\mathrm{sgn}(\Lambda)$ is divisible by 8. Kudla-Millson then define the composition
$$\theta: H^*_{cts}\left( O(\Lambda_\R), \mathscr S(\Lambda_\R) \right)\overset{\mathrm{res}} \to H^*(\Gamma, \mathscr S(\Lambda_\R))\overset{\Theta } \to H^*(\Gamma, \C)\simeq H^*(M,\C).$$
The cohomological theta correspondence is a morphism:
$$H^i_c(M,\C)\otimes H^{a-i}_{cts}\left( O(\Lambda_\R), \mathscr S(\Lambda_\R) \right)\to \mathscr C^\infty(Mp_2(\R))$$
$$(\mu, \nu) \mapsto \varphi(g) = \int_M \mu \wedge \theta(g\cdot \nu),$$
for $g\in Mp_2(\R)$. The action of $Mp_2(\R)$ on $H^*_{cts}\left( O(\Lambda_\R), \mathscr S(\Lambda_\R) \right)$ induces an action of its complexified Lie algebra, which splits into $\mathfrak h \oplus \mathfrak p_+ \oplus \mathfrak p_-$ (the Cartan, holomorphic, and anti-holomorphic parts). If $\nu$ is annihilated by $\mathfrak p_-$, and $\mathfrak h$ acts on $\nu$ with weight $w$, then $\varphi$ descends to a holomorphic modular form of weight $w/2$ for $Mp_2(\Z)$. If $m$ is even, then this modular form descends to $SL_2(\R)$. Kudla-Millson construct special classes $\nu$ such that the resulting $\varphi$ has only positive Fourier coefficients controlled by intersection numbers with the special cycles $C_{2n}$.

\end{remark}
\noindent
To apply this statement to our situation, we note that the further quotient map
$$g:D \to O(2k-2,10k-2)/O(2k-2)\times O(10k-2)$$
is a smooth proper fiber bundle with fiber $SO(2k-2)/U(k-1)$.  Given a positive definite real $(2k-2)$-plane $Z\subset \Lambda_\R$, a polarized Hodge structure is given by a choice of splitting
$$Z_\C\simeq H^{0,2} \oplus H^{2,0}\subset \Lambda_\C$$
into a pair of conjugate complex subspaces, isotropic with respect to the form.  A fiber of $g$ over $[Z]$ corresponds to this choice, which does not affect the orthogonal complement of $Z$.  Thus, the Noether-Lefschetz loci in $D$ are pulled back from the symmetric space:
$$\widetilde{\NL}_{2n}  = g^{-1} \left(\widetilde{C}_{2n}\right).$$
The constant term of the series can be interpreted in terms of the Hodge bundle on $\Gamma\bs D$.  Indeed, if $V$ is the tautological bundle of $p$-planes,
$$g^* V \otimes \C = V^{0,2} \oplus V^{2,0}$$
where each summand is isomorphic to $g^*V$ as a real vector bundle.  There is a natural complex structure on $V^{2,0}$ coming from the Hodge filtration, so we can take its Chern class:
$$c_{top}(V^{2,0}) = g^*e(V).$$
Theorem \ref{kmnum} is only proved for $\Gamma$ torsion-free, so that $M$ is a manifold.  For our application, we need to allow torsion elements, since ADE singularities have finite order monodromy.  For convenience, we will take $\Gamma = O(\Lambda)$.
\begin{lemma}  $O(\Lambda)$ contains a finite index normal subgroup which is torsion-free.\end{lemma}
\begin{proof} This is a well-known result of Selberg; see for instance Cor. 17.7 in \cite{borelalg}. \end{proof}\noindent
We denote the torsion-free subgroup by $\Gamma_{tf}\subset O(\Lambda)$. The analytic stack $[O(\Lambda) \bs D]$ can be realized as the quotient of a complex manifold by a finite group $G\simeq O(\Lambda)/\Gamma_{tf}$.  The spaces described above are related by
$$\xymatrix{
 & \ov{D}:= \Gamma_{tf} \bs D  \ar[rd]^g \ar[ld]_h & \\
[O(\Lambda) \bs D] & & M ,
}$$
where $h$ is a $G$-cover, and $g$ is a proper fiber bundle.  The modularity statement carries over to intersections on the stack as follows.  If $\alpha\in H_p([O(\Lambda)\bs D],\Q)$ is a rational homology class, and $\ov{\NL}_{2n}\subset \ov{D}$ is the Noether-Lefschetz locus in the manifold $\ov{D}$,
\begin{align*}
\alpha \cap h_*[\ov{\NL}_{2n}] &= h^*\alpha \cap [\ov{\NL}_{2n}]\\
 &= h^*\alpha\cap g^*[C_{2n}]\\
 &= g_*h^*\alpha\cap [C_{2n}],
\end{align*}
by repeated applications of the push-pull formula, which is valid for smooth stacks of Deligne-Mumford type in the sense of Behrend \cite{cohomology-of-stacks}.  Since the locus $\ov{\NL}_{2n}$ inside $\ov{D}$ is $O(\Lambda)$-invariant, its pushforward under $h$ acquires a multiplicity of $|G|$.  This overall factor can be divided out from the generating series.  With these adjustments, we modify Theorem \ref{kmnum} to fit our Hodge theoretic situation:
\begin{thm}\label{kmnum2}  For any homology class $\alpha \in H_p([O(\Lambda)\bs D],\Q)$, the series
$$\varphi(q) = \alpha\cap c_{top}(\lambda^\vee) + \sum_{r=1}^\infty \alpha \cap [\NL_{2n}] \,q^n$$
is a modular form of weight $6k-2$ and level $SL_2 (\Z)$. The constant term is the integral of the top Chern class of the dual Hodge bundle.\end{thm}\noindent
We apply this statement to the $\alpha = \mathfrak j_*[\F(Y)]$, as defined in Section \ref{simulres}.  To compute the intersection product, we spread out the period map to a section of a smooth fiber bundle over $\F(Y)$.  Using the principal $O(\Lambda)$-bundle $\mathfrak E\to \F(Y)$ in the definition of $\mathfrak j$, we set
$$\ov{D}(\LL)  := \ov{D} \times _{O(\Lambda)}\mathfrak E  = [(\ov{D}\times \mathfrak E)/O(\Lambda)], $$
which admits a section $s:\F(Y) \to \ov{D}(\LL)$ coming from the graph of the period map.  The Noether-Lefschetz loci can be spread out similarly
$$\ov{\NL}_{2n}(\LL) := \ov{\NL}_{2n} \times _{O(\Lambda)}\mathfrak E .$$
By Lemma 2.1 of \cite{km}, after further shrinking $\Gamma_{tf}$, we may assume that $\ov{\NL}_{2n}\subset \ov{D}$ has only normal crossing singularities. As a result, the map from the normalization $\ov{\NL}_{2n}(\LL)^{\mathrm{norm}}\to \ov{D}(\LL)$ is a local regular embedding.  The section $s:\F(Y)\to \ov{D}(\LL)$ is also local regular embedding of stacks, and it fits into the 2-Cartesian square:
$$\xymatrix{
W \ar[r]^{s'\,\,\,\,\,\,\,\,\,\,\,\,\,\,} \ar[d]_{g'} & \ov{\NL}_{2n}(\LL)^{\mathrm{norm}} \ar[d]^g\\
\F(Y) \ar[r]^s & \ov{D}(\LL).
}$$
The desired intersection number is given by
\begin{align*}
\mathfrak j_*[\F(Y)] \cap [\NL_{2n}] &= s_*[\F(Y)] \cap [\ov{\NL}_{2n}(\LL)^{\mathrm{norm}}] \\
 &= \deg s^! [\ov{\NL}_{2n}(\LL)^{\mathrm{norm}}]\\
 &= \deg g^! [\F(Y)].
\end{align*}
We use Vistoli's formalism \cite{vist}, to write the Gysin map $g^!$ in terms of a normal cone. Let $N=(s')^* N_g$, a vector bundle containing the normal cone $C_{W/\F(Y)}$. Since $W\subset \F(Y)$ is a closed substack of a Deligne-Mumford stack, and $N$ is the dual Hodge bundle on $W$, we are in the algebraic setting of Vistoli, so the intersection number is given by $0^!_N [C_{W/\F(Y)}]$.
This can be computed \'{e}tale locally on $\F(Y)$.  If $f:Z \to \F(Y)$ is an \'{e}tale morphism from a scheme, we can form the fibered product
$$\ov{D}(\LL)_Z := Z \us{\F(Y)}\times \ov{D}(\LL) \to Z$$
which also admits a section $s_Z:Z \to \ov{D}(\LL)_Z$.  Lastly, we form $W_Z = Z\times_{\F(Y) }W$ and $\ov{\NL}_{2n}(\LL)_Z = \ov{D}(\LL)_Z \times_{\ov{D}(\LL)}\ov{\NL}_{2n}(\LL)$.  This spaces fit into the cubic diagram below, where each face is 2-Cartesian.
$$
\xymatrix{
W_Z \ar[dd] \ar[rr] \ar[rd]_{f'} && \ov{\NL}_{2n}(\LL)_Z \ar'[d]^{g_Z}[dd] \ar[rd] &  \\
 & W \ar[dd] \ar[rr] &&  \ov{\NL}_{2n}(\LL) \ar[dd]^g  \\
Z \ar'[r][rr]^{s_Z\,\,\,\,\,\,\,\,\,\,\,\,}  \ar[rd]_f &&  \ov{D}(\LL)_Z \ar[rd]\\
& \F(Y) \ar[rr]^s && \ov{D}(\LL)
}
$$
The front and lateral sides are 2-Cartesian by construction, the bottom uses the fact that $s$ is a monomorphism, and the top and back are proven by repeated application of the universal property.  All the diagonal morphisms are \'{e}tale, by stability of the \'{e}tale property under base change.  The normal cones are related by
$$C_{W_Z/ Z} = (f')^* C_{W/\F(Y)}$$
Assuming that $Z$ covers the support of $W$ in $\F(Y)$, we have
$$\deg g_Z^![Z] = \deg(f)\cdot \deg g^![\F(Y)]. $$
We use this formula in Section \ref{count} to compute the contributions of 0-dimensional intersections to the $\Theta(q)$ correction term.

\section{Discriminant Hypersurfaces}\label{discsection}\noindent
Consider the cuspidal hypersurface cut out by $\Delta = 4A^3+27B^2$ of degree $12k$ inside $\P^{m+1}$.  The intersection multiplicity of a line $L\subset \P^{m+1}$ with $\Delta$ dictates the singularities in the surface $\pi^{-1}(L)$. Mildly singular elliptic surfaces still have pure Hodge structures (given by a minimal resolution), and the presence of exceptional curves makes these surfaces Noether-Lefschetz special. The contributions of such singular surfaces to the modular form $\varphi(q)$ are enough to determine it uniquely. In this section, we compute those contributions using classical enumerative geometry.
\begin{prop}\label{adeonly}
The fibers of $\nu:\mathscr S \to F(Y)$ have isolated rational double points of type $A_\rho$.
\end{prop}
\begin{proof}
The surface $\pi^{-1}(L)$ can only be singular at the singular points of the cubic fibers, since elsewhere it is locally a smooth fiber bundle.
\begin{itemize}
\item If $L$ intersects $\Delta_{sm}$ at a point $p$ with multiplicity $\mu$, then the local equation of $\pi^{-1}(L)$ near the node in the fiber $\pi^{-1}(p)$ is 
$$x^2+y^2=t^\mu.$$
This is an $A_{\mu-1}$ singularity\footnote{By convention, $A_0$ means a smooth point.}.  We refer to such lines as Type I.
\item If $L$ intersects $\Delta_{sing}=(A)\cap (B)$ at a point $p$, then the fiber $\pi^{-1}(p)$ has a cusp.  The local equation of $\pi^{-1}(L)$ there depends on the intersection multiplicity $\alpha$, resp. $\beta$, of $L$ the hypersurface $(A)$, resp. $(B)$:
$$x^3 +y^2 = t^\alpha x + t^\beta.$$
This singularity can be wild for large values of $\alpha$ and $\beta$, but the codimension of this phenomenon is $\alpha+\beta - 1$.  For $k\leq 4$, only $A_1$ and $A_2$ singularities can occur for $L\in F(Y)$ since the latter has dimension $k-1$.  We refer to such lines as Type II.
\end{itemize}
\begin{center}
\begin{tabular}{c|c|c}
$\alpha$ & $\beta$ & Type \\
\hline
$n\geq 1$ & 1 & $A_0$ \\
1 & $n>1$ & $A_1$ \\
2 & 2 & $A_2$
\end{tabular}
\end{center}
Lastly, we must rule out the possibility of lines $L\subset Y$ lying completely inside $\Delta$, because otherwise $\pi^{-1}(L)$ would not be normal.  For this, consider the incidence correspondence
$$\Omega = \{ (L,[A:B]): (4A^3+27B^2)|_L = 0 \} \subset F(Y)\times W(\P^{m+1},\O(k)).$$
The first projection $\Omega \to F(Y)$ is surjective and has irreducible fibers.  To see this, note that if $(4A^3+27B^2)|_L=0$, then
\begin{align*}
A|_L &= -3f^2,\\
B|_L &= 2f^3
\end{align*}
for some $f\in H^0(\P^1,\O(2k))$, by unique factorization.  The codimension of this locus in $W(\P^1,\O(k))$ is greater than
$$8k  > k-1 = \dim F(Y).$$
Hence, the second projection is not dominant, so for general $A$ and $B$, there are no lines on $Y$ contained inside the discriminant hypersurface.
\end{proof}\noindent
We introduce tangency schemes to record how these $A_\rho$ singularities appear, 
\begin{defn}
Given a partition $\mu$ of $12k$, let
$$T_\mu(\Delta) := \ov{\{L: \Delta_{sm}\cap L = \sum \mu_j p_j \}} \subset \G(1,m+1).$$
\end{defn}
\begin{prop}
For general $A$ and $B$, $T_\mu(\Delta)$ has codimension
$$\sum_{j=1}^l (\mu_j-1)$$
when the latter is $\leq k\leq 4$.
\end{prop}
\begin{proof}
For the partition $(\mu_1,1,\dots,1)$, consider the incidence correspondence
$$\Omega_{\mu_1} = \{ (L,p,[A:B]):\,\, p\in L,\,\mult_p(A^3+B^2)|_L\geq \mu_1 \} \subset U \times W(\P^{m+1},\O(k)),$$ 
The first projection $\Omega_{\mu_1}\to U$ has fiber cut out by $\mu_1$ equations on $W$, which we wish to be independent.  Setting $A(t)=a_0+a_1t+\dots$ and $B(t) = b_0+b_1t+\dots$ for $t$ the uniformizer at $p$, the differential of the multiplicity condition is given by
$$\begin{pmatrix} 3a_0^2 & 0 & 0  & 0 & \dots & 2b_0 & 0 & 0 & \dots & 0 \\
6a_0a_1 & 3a_0^2 & 0 & 0 & \dots & 2b_1  & 2b_0 & 0 & \dots & 0 \\
6a_0a_2+3a_1^2 & 6a_0a_1 & 3a_0^2 & 0 & \dots & 2b_2 & 2b_1 & 2b_0 & \dots & 0 \\
\dots & & & & & \dots  \end{pmatrix}$$
The rows are independent unless $a_0=b_0=0$.  In this case, comparing the $t$-valuations of $A$, $B$, and $A^3+B^2$, we see that $A(t)$, resp. $B(t)$, is actually divisible by $t^{ \lceil \mu_1/3 \rceil}$, resp. $t^{\lceil \mu_1/2\rceil}$, when $\mu_1\leq 6$.  There is an entire irreducible component
$$\Omega_{\mu_1,\rm{II}} = \{ (L,p,A,B):\,\, p\in L,\, \mult_p(A)\geq  \lceil \mu_1/3 \rceil,\, \mult_p(B) \geq  \lceil \mu_1/2 \rceil \}\subset \Omega$$
whose fiber a pair $(L,p)$ is a linear subspace of $W$.  For $\mu_1<6$, $\Omega_{\mu_1}$ is equidimensional with two irreducible components of dimension $\dim \G+1+\dim W - \mu$:
$$\Omega = \Omega_{\mu_1,\rm{I}}\cup \Omega_{\mu_1,\rm{II}}.$$
Note that $\Omega_{\mu_1,\rm{II}}$ consists of lines with points which lie in the {\it singular locus} of $\Delta$, where $A=B=0$. Since we are ultimately interested in $T_{\mu}(\Delta)$, which is defined in terms of tangency at {\it smooth} points of $\Delta$, we focus on $\Omega_{\mu_1,\rm{I}}$. Consider the second projection $\Omega_{\mu_1,\rm{I}}\to W$.  It is dominant by a dimension count, so the general fiber must have codimension $\mu_1-1$.\\\\
The case of a general partition $\mu$ is similar; gather only the multiplicities $\mu_j$ ($1\leq j\leq l$) which are greater than 1, and consider
$$\Omega_\mu = \{ (L,p_1,\dots,p_l,A,B):\,\,p_j\in L,\, (A^2+B^3)|_L =  \sum \mu_j p_j \}$$
$$\subset (U\times_\G \dots \times_\G U ) \times W(\P^{m+1},\O(k)) =: \mathscr U\times W.$$
Consider the fiber at $(L,p_1,\dots,p_l)\in \mathscr U$ for distinct points $p_j\in L$.  If $A$ and $B$ do not simultaneously vanish at any of the points, then the matrix of differentials can be row reduced to a matrix with blocks of the form
$$\frac{1}{n!} \partial_x^n(1,x,x^2,x^3,\dots)|_{x=p_j};\,\,\, (0\leq n\leq \mu_j).$$
This is a generalized Vandermonde matrix which has independent rows since the points are distinct.  If $A$ and $B$ vanish simultaneously at some $p_j$, then in fact they vanish maximally, which is a linear condition on $W$ of the expected codimension.  Hence, the general fiber of $\Omega_\mu \to \mathscr U$ has $2^l$ irreducible components, which collapse over the big diagonal.  Since the monodromy is trivial, we conclude that $\Omega_\mu$ itself has $2^l$ components.  We are only interested in one of them, $\Omega_{\mu,I}$, which gives the desired codimension for $T_\mu(\Delta)$.
\end{proof}\noindent
Next, we will show that Type II lines $L\in F(Y)$ are always limits of Type I lines (when $k\leq 4$), so we can effectively ignore them.  Trailing 1's in the partitions are suppressed for convenience.
\begin{lemma}\label{offending}
Let $J_{\alpha,\beta}(A,B)\subset \G(1,m+1)$ be the locus lines $L$ meeting $(A)$, resp. $(B)$, with multiplicity $\alpha$, resp. $\beta$, at a common point $p\in (A)\cap (B)$.  Then we have
\begin{align*}
J_{1,2}(A,B), J_{1,3}(A,B) &\subset T_2(\Delta);\\
J_{2,2}(A,B) &\subset T_3(\Delta).
\end{align*}
These are the only Type II lines which remain after intersecting with $F(Y)$.
\end{lemma}
\begin{proof}
Since $J_{1,2}(A,B)$ is codimension 2, we intersect $\Delta$ with a general $\P^2\subset \P^{m+1}$ containing $p$.  The Type II lines are those which pass through the cusps of $\P^2\cap\Delta$ in the preferred (cuspidal) direction.  They lie in the dual plane curve to $\P^2\cap \Delta$.  Since $J_{2,2}(A,B)$ and $J_{1,3}(A,B)$ are codimension 3, we intersect $\Delta$ with a general $\P^3\subset \P^{m+1}$ containing $p$.  The lines in $J_{2,2}(A,B)$ are those tangent to the curve $\P^3\cap (A)\cap (B)$.  They are always limits of flex lines at smooth points of $\P^3\cap\Delta$, by a local calculation. The lines in $J_{1,3}(A,B)$ are those meeting $\P^3\cap (A)\cap (B)$, while being flex to $\P^3\cap (B)$. They are always limits of tangent lines at smooth points of $\P^3\cap \Delta$, by a local calculation.
\end{proof}\noindent
With these dimension results in hand, we turn to degree computations. The main tools are the Pl\"{u}cker formulas relating degree, \# of nodes, and \# of cusps for a plane curve ($d,\delta,c$) with those three numbers for the dual curve ($d^*, \delta^*, c^*$). Note that $\delta^*$ (resp. $c^*$) is the number of bitangent (resp. flex) lines to the original curve.
\begin{prop}\label{schubert1}
The class of $T_2(\Delta)$ in $H_{4m-2}(\G(1,m+1))$ is Poincar\'{e} dual to
$$12k(6k-1)\cdot \sigma_1 \in H^2(\G(1,m+1)).$$
\end{prop}
\begin{proof}
We intersect $[T_2(\Delta)]$ with the complementary Schubert class $\sigma^1$, which is represented by a pencil of lines in a general $\P^2\subset \P^{n+1}$.  Since $\P^2\cap \Delta$ is a curve of degree $d=12k$ with $c = 4k\cdot 6k$ cusps, the intersection number is given by the Pl\"{u}cker formula for the dual degree:
$$d^* = d(d-1) - 3c.$$
\end{proof}
\begin{prop}
The class of $T_3(\Delta)$ in $H_{4m-4}(\G(1,m+1))$ is Poincar\'{e} dual to
$$24k(10k-3)\cdot \sigma_{11} + 24k(6k-1)(4k-1)\cdot \sigma_2\in H^4(\G(1,m+1)).$$
\end{prop}\label{schubert2}
\begin{proof}
First, we intersect $[T_3(\Delta)]$ with the class $\sigma^{11}$, which is represented by all the lines in a general $\P^2\subset \P^{n+1}$.  The intersection number is again given by the Pl\"{u}cker formula for flex lines applied to $\P^2\cap \Delta$:
$$c^* = 3d(d-2) - 8 c.$$
Next, we intersect with the class $\sigma^2$, which is represented by all lines through a point in a general $\P^3\subset \P^{n+1}$.  To compute this number we find the top Chern class of a bundle of principal parts.  Consider the universal line $U \to \G(1,3)\times \P^3$, and let $\mathcal P$ denote the bundle whose fiber at a point $(L,p)$ is the space of 2nd order germs at $p$ of sections of $\O_L(12k)$.  The surface $\P^3\cap \Delta$ induces a section of $\mathcal P$, by restriction, vanishing at each flex.  By standard arguments in \cite{he}, $\mathcal P$ admits a filtration with successive quotients given by
$$\pi_2^* \O_{\P^3}(12k),\, \pi_2^* \O_{\P^3}(12k)\otimes \Omega^1_{U/\G},\, \pi_2^* \O_{\P^3}(12k) \otimes \Sym^2 \Omega^1_{U/\G}.$$
By the Whitney sum formula, its top Chern class is given by
\begin{align*}
c_3(\mathcal P) &= (12k\zeta) ((12k-2)\zeta + \sigma_1) ((12k-4)\zeta + 2\sigma_1)\\
 &= 96 k(6k-1)(3k-1)\zeta^3 + 48k(9k-2)\zeta^2 \sigma_1 + 24k\zeta \sigma_1^2,
\end{align*}
where $\zeta$ denotes the relative hyperplane class on $U\to \G(1,3)$ as a projective bundle.  We intersect this class with $\sigma_2\in H^4(\G(1,3))$ to get the number of flex lines for $\Delta$ passing through a general point in $\P^3$.
$$c_3(\mathcal P)\cdot \sigma_2 = 96k(6k-1)(3k-1) + 48k(9k-2) + 24k.$$
This number is too high, since lines meeting $(A)\cap (B)$ tangent to $(B)$ have vanishing principal part, and each contribute multiplicity 8 by a local calculation.  If $\theta: (B)\to \P^{3*}$ is the Gauss map associated the hypersurface $(B)$, the number of such lines is
$$\deg\left(\theta|_{(A)\cap (B)}\right) = 4k\cdot 6k(6k-1).$$
Subtracting this cuspidal correction from the Chern class gives the desired number.
\end{proof}
\begin{prop}\label{schubert3}
The class of $T_{2,2}(\Delta)$ in $H_{4m-4}(\G(1,m+1))$ is Poincar\'{e} dual to
$$108k (3k-1)(8k^2-1) \cdot \sigma_{11} + 36k (6k-1)(4k-1)(3k-1)\cdot \sigma_2\in H^4(\G(1,m+1)).$$
\end{prop}
\begin{proof}
The intersection of $[T_{2,2}(\Delta)]$ with the class $\sigma^{11}$ is given by the Pl\"{u}cker formula for bitangent lines applied to $\P^2\cap \Delta$:
$$\delta^* = \frac{ d^*(d^*-1) -d- 3c^*}{2}.$$
Next, we intersect with the class $\sigma^2$, which counts bitangent lines to $S=\P^3 \cap \Delta$ passing through a general point $p\in \P^3\subset \P^{n+1}$.  Consider the projection from $p$
$$\Pi_p: S \to \P^2.$$
The normalization of $S$ is a smooth surface $\widetilde{S}$, and we write
$$\widetilde{\Pi}_p: \widetilde{S} \to \P^2$$
for the composition.  The ramification curve $R=R'\cup R''$ has two components:  $R'$ is the pre-image of the cusp curve, and $R''$ is the closure of the ramification locus for $\Pi_p:S_{sm} \to \P^2$.  These lie over the branch locus $B=B'\cup B''\subset \P^2$.  The degree of $B''$ was already computed in Proposition \ref{schubert1}:
$$\deg(B'') = 12k(6k-1),$$
so we know its arithmetic genus.  Nodes of $B''$ correspond to a bitangent lines to $S$ through $p$, and cusps of $B''$ correspond to a flex lines to $S$ through $p$, which we already counted in Proposition \ref{schubert2}.  There are no worse singularities in $B''$ for a general choice of $p$, so it suffices to compute
$$\delta(B'') = p_a(B'') - g(R'').$$
The Riemann-Hurwitz formula says that
$$K_{\widetilde{S}} = \widetilde{\Pi}_p^* K_{\P^2}+ R = -3H+R.$$
Realizing $\widetilde{S}$ inside the blow up of $\P^3$ along $\Sigma = \P^3\cap (A)\cap (B)$, the adjunction formula reads
$$K_{\widetilde{S}} = (12k-4)H-E = (12k-4)H - 2R'.$$ 
Combining these equations, we deduce that
$$R'' = (12k-1)H - 3R'.$$
This is enough to determine the genus of $R''$, using
\begin{align*}
H\cdot H &= 12k; \\
H\cdot R' &= 24k^2;\\
R'\cdot R' &= 48k^3.
\end{align*}
The latter can be computed inside the projective bundle $\P N_{\Sigma/\P^3}$, where $R'$ is the class of $\P N_{\Sigma/(B)}$.  As an additional check, observe that
$$R'\cdot R'' = 24k^2(6k-1) = \deg\left(\theta|_{(A)\cap (B)}\right),$$
which agrees with the cuspidal correction from Proposition \ref{schubert2}.  Finally, we use the genus formula on $R''$:
$$2g - 2 = (K_{\widetilde{S}} + R'')\cdot R';'$$
$$\delta(B'') = 12k(9k-1)(6k-1)(4k-1).$$
Subtracting the flex line count from Proposition \ref{schubert2} leaves the desired number.
\end{proof}\noindent
In the sequel, we will use the notation
$$t_\mu := [T_{\mu}(\Delta)] \cdot F(Y)$$
when this intersection is 0 dimensional, that is
$$\sum_{j=1}^l (\mu_j-1) = k-1$$
By the results of Section \ref{simulres}, the stack $\F(Y)$ will have isotropy group
$$\prod_{j=1}^l \mathfrak S_{\mu_j}$$
at these isolated points.

\section{Counting Curves}\label{count}\noindent
We begin by discussing the case $k=1$, which is trivial because the period domain is a point.  Since $F(Y)$ is 0-dimensional, let $N_m = \# F(Y)$, and we may assume that for each $[L]\in F(Y)$, the rational elliptic surface $S= \pi^{-1}(L)$ is smooth.  Every class in the polarized N\'{e}ron-Severi lattice $\NS_0(S) \simeq E_8$ is the orthogonal projection of a section curve.  The degree shifts from Lemmas \ref{shift1} and \ref{shift2} match, so we have
$$\sum_{n\geq 1} r_X(n)q^n = N_m \theta_{E_8}(q) - N_m$$
and $\theta_{E_8}(q)$ is the modular form of weight 4, as desired.\\\\
For $2\leq k\leq 4$, recall the family of elliptic surfaces $\nu:\mathscr S \to F(Y)$ defined by the diagram
$$\xymatrix{
  & \mathscr S \ar[d]^{\pi'} \ar[ld]_{\nu} \ar[r] & X \ar[d]^{\pi} \\
F(Y) & U \ar[l] \ar[r] & Y.
}.$$
The period map $j:F(Y) \to O(\Lambda) \bs D$ lifts to an immersion of stacks $j:\F(Y) \to [O(\Lambda)\bs D]$ of Deligne-Mumford type.  Theorem \ref{kmnum2} applied to $\alpha = j_*[\F(Y)]$ yields a classical modular form $\varphi(q)$ of weight $6k-2$ and full level.  Our first task is to determine this modular form.  We start by computing the constant term as the top Chern class of the dual Hodge bundle:
\begin{defn}
The Hodge bundle of the family $\nu:\mathscr S \to F(Y)$ is defined as the pushforward of the relative dualizing sheaf:
$$\lambda := \nu_*\left( \omega_{\mathscr S/F(Y)}\right).$$
\end{defn}
\begin{prop}
Let $S$ be the tautological rank 2 bundle on $F(Y)$.  Then
$$\lambda \simeq \Sym^{k-2}(S^\vee) \otimes \O_{F(Y)}(\sigma_1)$$
\end{prop}
\begin{proof}
Writing $\nu$ as the composition $u\circ \pi'$, we compute the pushforward in stages.
\begin{align*}
\pi'_*\left(\omega_{\mathscr S/F(Y)}\right) &= \pi'_*\left(  \omega_{\mathscr S/U}\otimes \pi'^* \omega_{U/F(Y)}  \right )\\
 &=  \pi'_*(  \omega_{\mathscr S/U} )\otimes\omega_{U/F(Y)} \\
 &= \O_U(k\zeta) \otimes \omega_{U/F(Y)}\\
 &= \O_U(k\zeta) \otimes \O_U(-2\zeta)\otimes u^*\O_{F(Y)} (\sigma_1)
\end{align*}
using the fact that $\pi'$ is a Weierstrass fibration, and $u$ is the restriction to $F(Y)$ of the projective bundle $\P(S)\to \G(1,m+1)$.  Next, we compute
$$u_* (\O_U((k-2)\zeta) \otimes u^*\O_{F(Y)} (\sigma_1)) = \Sym^{k-2}(S^\vee)\otimes \O_{F(Y)} (\sigma_1) $$
\end{proof}\noindent
The top Chern class of $\lambda$ is enough to determine the constant term of $\varphi(q)$ when $k=2$.  Indeed,
\begin{prop}
For $k=2,3,4$, the space of modular forms $\Mod(6k-2,SL_2(\Z))$ has dimension $1,2,2$, respectively.
\end{prop}
\begin{proof}
This follows from the presentation of the ring $\Mod(\bullet, SL_2(\Z))$ as a free polynomial ring on the Eisenstein series $E_4$ and $E_6$.
\end{proof}\noindent
The positive degree terms of $\varphi(q)$ are Noether-Lefschetz intersections.  By the split sequence of Shioda-Tate, there are two sources of jumping Picard rank in the family of Hodge structures.  
\begin{itemize}
\item Resolved singular surfaces have non-trivial $V_0(S)$.  If $e$ is the class of an exceptional curve, then $\langle f,z,e\rangle$ has intersection matrix
$$\begin{pmatrix}  0 & 1 & 0 \\ 1 & -k & 0 \\ 0 & 0 & -2 \end{pmatrix}.$$
\item Surfaces with extra sections have non-trivial Mordell-Weil group $\MW(S/\P^1)$.  If $\sigma$ is the class of a section, then $\langle f,z,\sigma\rangle$ has intersection matrix
$$\begin{pmatrix}  0 & 1 & 1 \\ 1 & -k & z\cdot\sigma \\ 1 & z\cdot\sigma & -k \end{pmatrix}.$$
\end{itemize}
By the shift in Lemma \ref{shift1}, the Mordell-Weil jumping starts to contribute at order $q^k$, which matches the shift in Lemma \ref{shift2} for the rational curve class on $X$.  Thus, terms of order $<k$ are determined by enumerating the singular surfaces.
\begin{remark}  Every vector in the sublattice $\MW(S/\P^1)$ corresponds to the class of some section curve.  On the other hand, classes in $V_0(S)$ are less geometric:  they are arbitrary $\Z$-linear combinations of exceptional curve classes, which are accounted for by the theta series $\Theta(q)$.\end{remark}\noindent
For $k=3$, we can compute the full generating series.  The $q^1$ term is given by
$$[\varphi]_1 = \mathfrak j_*[\F(Y)]\cdot [\NL_2],$$
which is an excess intersection along $T_2(\Delta)\cap F(Y)$.  If $v\in \Lambda$ lies in $Z^\perp$ with $v^2=-2$, then any integer multiple $mv\in \Lambda$ lies in $Z^\perp$ with $(mv)^2=-2m^2$, and the corresponding component of $\NL_{2m^2}$ meets $\F(Y)$ with isomorphic normal cone.  As a result this excess intersection contributes $\theta_1(q)$ to the generating series $\varphi(q)$.  Next, the $q^2$ term is given by
$$[\varphi]_2 = \mathfrak j_*[\F(Y)]\cdot [\NL_4],$$
which is a 0-dimensional intersection along $T_{2,2}(\Delta)\cap F(Y)$. Since $\F(Y)$ has isotropy group $\mathfrak S_2\times \mathfrak S_2$ there, we can compute
$$[\varphi]_2 = \frac{1}{4} t_{2,2}.$$
This completely determines the modular form $\varphi(q)$, so we can solve for $[\varphi]_1$.  The only remaining singular surfaces which contribute are the 0-dimensional intersection along $T_3(\Delta)\cap F(Y)$.  Since $\F(Y)$ has isotropy group $\mathfrak S_3$ there, we have
$$\varphi(q) = \frac{1}{2}[\varphi]_1 \theta_1(q) + \frac{1}{4}t_{2,2} \left(\theta_1(q)^2 - 2\theta_1(q)\right) + \frac{1}{6} t_3 \left(\theta_2(q) - 3\theta_1(q)\right) + \sum_{n\geq 3} r_X(n) q^n$$
For each successive singular stratum, we subtract the root lattice vectors which are limits of previous strata.  There are 3 copies of $A_1$ in $A_2$, and there are 2 copies of $A_1$ in $A_1\times A_1$, so we subtract the double counted exceptional classes.\\\\
For $k=4$, there are too many undetermined excess intersections (denoted $a_i$) to determine $\varphi(q)$, but we still have the general form
$$\varphi(q) = \Theta(q) + \sum_{n\geq 4} r_X(n)q^n,$$
where the theta correction term is given by:
\begin{align*}
\Theta(q) =&\,\,\,\,\,\, a_1 \theta_1(q) + a_2 \left(\theta_1(q)^2 - 2\theta_1(q)\right) + a_3 \left(\theta_2(q) - 3\theta_1(q)\right) \\
 &+ t_4 \left(\theta_3(q) - 4\theta_2(q) - 3 \theta_1(q)^2 + 18\theta_1(q)\right) \\
 &+ t_{2,2,2} \left(\theta_1(q)^3 - 3\theta_1(q)\right)\\
 &+ t_{2,3} \left(\theta_1(q)\theta_2(q) - 4\theta_1(q)\right).
 \end{align*}
We conclude with two examples of the full computation:
\begin{ex}
$m=d=2.$
\end{ex}\noindent
The quadric surface $Y\subset \P^3$ is isomorphic to $\P^1\times \P^1$, and the Weierstrass model $X$ is a Calabi-Yau threefold sometimes called the STU model.  Its curve counts were computed previously, and can be found in \cite{kmps}.
\begin{align*}
\varphi(q) &= -2E_4E_6;\\
\Theta(q) &= -266+264\, \theta_1.
\end{align*}

\begin{ex}
$m=2$, $d=1$.
\end{ex}\noindent
The hyperplane $Y\subset \P^3$ is isomorphic to $\P^2$, and the Weierstrass model $X$ is a Calabi-Yau threefold, which can also be realized as the resolution of the weighted hypersurface
$$X_{18}\subset \P(1,1,1,6,9).$$
The topological string partition function of $X$ is computed in \cite{hkk}, but our formula is the first in the mathematical literature.
\begin{align*}
\varphi(q) &= \frac{31}{48} E_4^4 + \frac{113}{48} E_4 E_6^2; \\
\Theta(q) &= 47253-93582\, \theta_1 +46008\,\theta_1^2 + 324\, \theta_2.
\end{align*}

\begin{ex}
$m=3$, $d=3$.
\end{ex}\noindent
The hypersurface $Y\subset \P^4$ is a cubic threefold, and the Weierstrass model $Y$ is a (non-CY) fourfold.
\begin{align*}
\varphi(q) &= \frac{433}{16}E_4^4 + \frac{1439}{16}E_4 E_6^2;\\
\Theta(q) &= 2089215-4107510\, \theta_1 +1969272\,\theta_1^2+49140\,\theta_2.
\end{align*}

\section{Future Directions}\noindent
Theorem \ref{main} gives a generating series for true counts of smooth rational curves on $X$ lying over lines in $Y$.  One can ask how this relates to the generating series for genus $0$ Gromov-Witten invariants, which virtually count nodal rational curves.  When $X$ is a threefold, the work of Oberdieck-Shen \cite{oberdieck} on stable pairs invariants of elliptic fibrations implies via the GW/pairs correspondence \cite{pix} that
$$\sum_{n=0}^\infty \text{GW}^X_0(\ell+nf) q^n = \varphi(q)\cdot\eta(q)^{-12k},$$
where $\eta(q)$ is the Dedekind eta function.  We can recover this formula in the case $k=2$, and for general $k$ the $\Theta(q)$ correction term from Theorem \ref{main} yields formulas for the local contributions of $A_\rho$ singular surfaces $S\subset X$ to the stable pairs moduli space; compare with \cite{toda}.\\\\
Counting curves in base degree $e>1$ presents challenges involving degenerations of Hodge structures.  The analog of the family $\nu$ in this situation is
$$\xymatrix{
  & \mathscr S \ar[d]^{\pi'} \ar[ld]_{\nu} \ar[r] & X \ar[d]^{\pi} \\
\ov{M}_0(Y,e) & \mathcal C \ar[l] \ar[r] & Y.
},$$
where the Fano variety is replaced by the Kontsevich space of stable maps.  At nodal curves
$$C_1\cup C_2 \to Y,$$
the fiber of $\nu$ is a surface with normal crossings:
$$S = \pi^{-1}(C_1) \cup \pi^{-1}(C_2),$$
which does not have a pure Hodge structure.  Thus, the associated period map
$$\ov{M}_0(Y,e) \dashrightarrow \Gamma \bs D$$
does not extend over all of $\ov{M}_0(Y,e)$.  Since the Noether-Lefschetz loci in $\Gamma\bs D$ are non-compact, we must compactify the period map to a larger target in order to have a topological intersection product.  Such a target $(\Gamma\bs D)^*$ is provided in \cite{gglr}, which satisfies the Borel Extension property for all period maps coming from algebraic families.  The boundary of the partial completion $(\Gamma \bs D)^*$ consists of products of period spaces of lower dimension, whose Noether-Lefschetz classes satisfy a modularity statement.  Motivated by this observation and computations in the Hermitian symmetric case, we make a conjecture for higher base degrees.\\\\
Let $X$ be a Weierstrass fibration in $W(Y,\O(k))$, where $Y\subset \P^{m+1}$ is a hypersurface of degree $d$, and
$$k= \left(\frac{e+1}{e}\right) m - d +\left( 2 - \frac{2}{e}\right).$$
Note:  $m\equiv 2$ (mod $e$) is equivalent to integrality of this expression.
\begin{conj}
Let $r_X(e,n)$ be the number of smooth rational curves on $X$ in the homology class $e\ell+nf$.  Then for $ke\leq 4$ or $m=2$,
$$\sum_{n\geq 1} r_X(e,n)q^n = \varphi(q) - \Theta(q),$$
where $\varphi(q)\in \mathrm{QMod}(\bullet,SL_2(\Z))$, and $\Theta(q)\in \Q[\theta_1,\theta_3,\theta_3]_{<ke}$.
\end{conj}
\noindent
Recall that {\it quasi-modular forms} are an enlargement of the algebra of modular forms to include the Eisenstein series $E_2$:
$$\mathrm{QMod}(\bullet,SL_2(\Z)) = \Q[E_2,E_4,E_6].$$

\newpage

\section*{Appendix A:  Torsion in the Mordell-Weil group}\noindent
Let $S\to \P^1$ be a regular minimal elliptic surface.  The vertical sublattice
$$V_0(S) \subset \NS_0(S)$$
is a direct sum of ADE root lattices, one for each singular fiber.  The discriminant group $d_i$ of each root lattice is the component group of the N\'{e}ron model for the degeneration.  A torsion element of $\MW(S/\P^1)$ restricts to a non-identity component on some fiber, so we have an embedding of finite abelian groups
$$\mathrm{TMW}(S/\P^1) \hookrightarrow \bigoplus_i d_i$$
Furthermore, $\mathrm{TMW}(S/\P^1)$ is totally isotropic with respect to the quadratic form on the discriminant group.  For the surfaces appearing in this paper (with $k\leq 4$), the discriminant group is sufficiently small that there are no non-trivial isotropic subgroups, so $\MW(S/\P^1)$ is torsion-free.

\section*{Appendix B:  Configurations of points on a line}\noindent
We study configurations of $6k$ unordered points on $\P^1$.  The moduli space of point configurations is given by
$$M_{6k} := \P^{6k}/PGL(2).$$
We will often ignore phenomena in codimension $\geq k$, since $\dim F(Y)=k-1$ and $F(Y)\subset \G(1,m+1)$ is freely movable.  Away from codimension $k$, there are at least 
$$6k-2(k-1) = 4k+2$$
singleton points.  In particular, all such configurations are GIT-stable.  For a fixed general hypersurface $B\subset \P^{m+1}$ of degree $6k$, we have a morphism
$$\phi_B:\G(1,m+1) \to M_{6k},$$
given by intersecting with $B$.  First we show that when $m$ is large, $\phi_B$ has large rank.
\begin{lemma}\label{highrank}
If $2m\geq 4k$, then $d\phi$ has rank $\geq 2k$ at lines $L$ meeting $B$ in $\geq 4k+2$ reduced points.
\end{lemma}
\begin{proof}
Consider the incidence correspondence
$$\Omega := \{ (B,L): \mathrm{rank}(d\phi_B)<2k \}\subset \P^N\times \G(1,m+1).$$
The second projection $\Omega \to \G(1,m+1)$ is dominant, and we study the fiber of this morphism.  Assume that $B$ intersects $L$ transversely at $q,o,p_1,p_2,\dots,p_{4k}$.  Pick coordinates on $\P^{m+1}$ such that $\mathbb T_q B\simeq \P^m$ is the hyperplane at infinity, $o$ is the origin, and $T_{o}B$ is orthogonal to $L$.  The remaining points are nonzero scalars, so we assume that $p_1=1$.  Pick coordinates on $\G(1,m+1)$ near $L$ by taking pencils based at $q\in L$ (resp. $o\in L$) in a set of $m$ general $\P^2$'s containing $L$.  At each marked point $p_i$, the transverse tangent space $T_{p_i}B$ is given by some slope vector
$$\vec{\lambda}_i \in \C^m.$$
The $(4k-1)\times (2m)$ matrix for $d\phi$ restricted to these points is given by 
$$d\phi = \begin{pmatrix} p_2\lambda_1^1 - \lambda_2^1 & \lambda_1^1 - \lambda_2^1 & p_2\lambda_1^2 - \lambda_2^2 & \lambda_1^2 - \lambda_2^2 & \dots & p_2\lambda_1^m - \lambda_2^m & \lambda_1^m - \lambda_2^m  \\
p_3\lambda_1^1 - \lambda_3^1 & \lambda_1^1 - \lambda_3^1 & p_3\lambda_1^2 - \lambda_3^2 & \lambda_1^2 - \lambda_3^2 & \dots & p_3\lambda_1^m - \lambda_3^m & \lambda_1^m - \lambda_3^m  \\
\vdots & \vdots & \vdots & \vdots & \ddots & \vdots & \vdots \\
p_{4k}\lambda_1^1 - \lambda_{4k}^1 & \lambda_1^1 - \lambda_{4k}^1 & p_{4k}\lambda_1^2 - \lambda_{4k}^2 & \lambda_1^2 - \lambda_{4k}^2 & \dots & p_{4k}\lambda_1^m - \lambda_{4k}^m & \lambda_1^m - \lambda_{4k}^m  \end{pmatrix}.$$
The coordinates of $\vec{\lambda}_i$ are free, and the small rank variety is cut out by the $(2k)\times (2k)$ minors of this matrix.  This is a linear section of a determinantal variety, defined by a 1-generic matrix of linear forms in the sense of \cite{eisenbud}.  By the principal result of \cite{eisenbud}, it has the expected codimension:
$$c= (2k)(2m-2k+1)>2m.$$
Hence, the projection $\Omega \to \P^N$ is not dominant for dimension reasons.
\end{proof}\noindent
To understand the tangent space to $F(Y)$ inside $\G(1,m+1)$, we use the short exact sequence of normal bundles
$$0 \to N_{L/Y} \to N_{L/\P^{m+1}} \to N_{Y/\P^{m+1}}|_L \to 0$$
$$0 \to H^0( N_{L/Y}) \to H^0(N_{L/\P^{m+1}} ) \to H^0(N_{Y/\P^{m+1}} ).$$
The tangent space $T_{[L]} F(Y) \simeq H^0(N_{L/Y})$ can be identified with the kernel of
$$H^0(L,\O(1)^{\oplus m}) \to H^0(L,\O(d)).$$
This linear map can be understood as follows:  for a set of $m$ general $\P^2$'s containing $L$, consider $Y\cap \P^2_i = L\cup C_i$.  The residual $C_i\cap L$ is a section of $g_i\in H^0(L,\O(d-1))$, and together they give the map.  In coordinates, the matrix looks like
$$\begin{pmatrix} g_{11} & 0 & g_{21} & 0 &\dots & 0 \\
g_{12} & g_{11} & g_{22} & g_{21} & \dots & g_{m1} \\
\vdots & \vdots & \vdots & \vdots & \ddots & \vdots \\
g_{1d} & g_{1(d-1)} & g_{2d} & g_{2(d-1)} & \dots & g_{m(d-1)} \\
0 & g_{1d} & 0 & g_{2d} & \dots & g_{md}
\end{pmatrix}.$$
Only the residuals matter when determining $H^0(N_{L/Y})$, and any set of residuals forms $g_i\in H^0(L,\O(d-1))$ comes from a hypersurface $Y$ containing $L$.
\begin{prop}
When $2m\geq 4k$, the morphism
$$\mu_X: F(Y) \to M_{6k}$$
is an immersion for general $Y$ and $X\in W(Y,\O(k))$.
\end{prop}
\begin{proof}
Consider the incidence correspondence
$$\Omega := \{ (Y,L): L\subset Y,\, d\mu_X\text{ is not an immersion} \} \subset \P^N \times \G(1,m).$$
The second projection $\Omega \to \G(1,m+1)$ is surjective, and we study the fiber of this morphism.  The condition $Y\subset L$ is codimension $d+1$, so we need $k$ more independent conditions for the codimension to exceed $2m$.  For this, we use another incidence correspondence:  let $K = \ker(d\phi_{[L]})\subset \C^{2m}$ which has dimension $\leq  2m-2k$ by Lemma \ref{highrank}, and consider
$$\Omega' := \{ (M,w): w\subset \ker(M) \} \subset \P^{md-1}\times \P K,$$
where $\C^{md}\subset \Hom(\C^{2m},\C^{d+1})$ is the subspace of matrices which come from residual forms.  It suffices to prove that $\pi_1(\Omega')\subset \P^{md-1}$ has codimension $\geq k$.  The fibers of the second projection $\Omega'\to \P K$ are cut out by $d+1$ linear conditions in $md$ variables.  In terms of the coordinates of $w$, the conditions are
$$\begin{pmatrix} w_1 & 0 & 0 & 0 & \dots & w_3 & 0 & 0 & 0 & \dots \\
w_2 & w_1 & 0 & 0 & \dots & w_4 & w_3 & 0 & 0 & \dots \\
0 & w_2 & w_1 & 0 & \dots & 0 & w_4 & w_3 & 0 & \dots \\
 & & \ddots & \ddots &  & & & \ddots & \ddots
 \end{pmatrix}.$$
The degeneracy loci of this matrix are high codimension in $\P K$ by explicit calculation with minors, so we have the dimension count:
\begin{align*}
\dim \Omega' &= \dim \P K + (md-1)-(d+1)  \\
 &\leq (2m-2k-1) + (md-1) - (2m-k+1)\\
\codim\, \pi_1(\Omega') &\geq (2m-k+1) - (2m-2k-1) = k+2.
\end{align*}
\end{proof}
\begin{prop}
When $2m<4k$, the morphism
$$\mu_X:  F(Y) \to \W_k$$
is an immersion for general $Y$ and $X\in W(Y,\O(k))$.
\end{prop}
\begin{proof}
Since $2m-6k-1<2m-4k-1<0$, general forms $A\in H^0(\P^{m+1},4k)$ and $B\in H^0(\P^{m+1})$ do not vanish on any line.  Hence, we have a morphism
$$\phi_{A,B}: \G(1,m+1) \to M_{4k}\times M_{6k}$$
which we claim is an immersion on $F(Y)$.  Since $F(Y)\subset \G(1,m)$ is freely movable, we may assume that each line intersects $A$ and $B$ at $\geq 8k+2$ reduced points.  Consider the incidence correspondence
$$\Omega = \{([A:B],L): d\mu_{[L]} \text{ is not injective}\}\subset  W(\P^{m+1},\O(k))\times \G(1,m+1) .$$
The fiber of the second projection $\Omega \to \G(1,m+1)$ is a linear section of a determinantal variety, as in Lemma \ref{highrank}.  By \cite{eisenbud}, it has the expected codimension:
$$c=(8k+2)-2m+1 > 2m.$$
Hence, the projection $\Omega\to W(\P^{m+1},\O(k))$ is not dominant for dimension reasons.
\end{proof}

\bibliography{RLD}
\bibliographystyle{plain}

\end{document}